\documentclass{amsart}

\usepackage{amssymb}
\usepackage{color}
\newcommand{\R}{\mathbb R}
\newcommand{\N}{\mathbb N}
\newcommand{\cst}{\mbox{\textnormal{Cst}}}
\newcommand{\dsp}{\displaystyle}
\newtheorem{proposition}{Proposition}
\newtheorem{theorem}{Theorem}
\newtheorem{definition}{Definition}
\newtheorem{remark}{Remark}

\newtheorem{lemma}{Lemma}
\newtheorem{example}{Example}
\title[Variable depth KDV equations ]{Variable depth KDV equations and generalizations to more nonlinear regimes}
\author{Samer Israwi}
\address{Universit\'e Bordeaux I; IMB, 351 Cours de la Lib\'eration, 33405 Talence Cedex, France}
\email{Samer.Israwi@math.u-bordeaux1.fr}

\begin{document}

\maketitle

\begin{abstract}\label{abstract}
We study here  the water-waves problem for uneven bottoms in a highly nonlinear regime where
the small amplitude assumption of the Korteweg-de Vries (KdV) equation is enforced. It is known,
that for such regimes, a generalization of the KdV equation (somehow linked to 
the Camassa-Holm equation) can be derived and justified \cite{AD:08} when the bottom is
flat. We generalize here this result
with a new class of equations taking into account variable bottom topographies. Of course, the
many variable depth KdV equations existing in the literature are recovered as particular cases. 
  Various regimes for the topography regimes are investigated and we prove consistency 
of these models, as well as a full justification 
 for some of them. We also study the problem of wave breaking for our new 
 variable depth and highly nonlinear generalizations of the KDV equations.
\end{abstract}

\section{Introduction}\label{int}
\subsection{General Setting}\label{GS}
This paper deals with the water-waves problem for uneven bottoms, which consists in studying the motion of the 
free surface and the evolution of the velocity field of a layer of fluid under the following assumptions:
the fluid is ideal, incompressible, irrotationnal, and under the only influence of gravity.
 Earlier works have set a good theoretical background for this problem. Its well-posedness has been discussed among others by Nalimov \cite{nalimov}, Yasihara \cite{yoshihara}, Craig \cite{craig}, Wu \cite{wu1}, \cite{wu2} and Lannes \cite{lannes}. 
Nevertheless, the solutions of these equations are very difficult to describe, because of the 
complexity of these equations. At this point, a classical method is to choose an asymptotic 
regime, in which we look for approximate models and hence for approximate solutions. 
More recently Alvarez-Samaniego and Lannes \cite{AL} rigorously justified the relevance of
 the main asymptotical models used in coastal oceanography, including: shallow-water equations, 
Boussinesq systems, Kadomtsev-Petviashvili (KP) approximation, Green-Naghdi equations (GN),
 Serre approximation, full-dispersion model and deep-water equations. Some of these models
 capture the existence of solitary water-waves and 
the associated phenomenon of soliton manifestation \cite{Jo-b}. The most prominent example 
is the Korteweg-de Vries  (KdV) equation \cite{KdV}, the only member of the wider family of 
BBM-type equations \cite{BBM} that is integrable and relevant for the phenomenon of soliton manifestation.
The KDV approximation originally derived over flat bottoms has been rigorously justified in \cite{craig,sw,BCL,ig}. When the 
bottom is not flat, various generalizations of the KDV equations with non constant coefficients
have been proposed \cite{kirby,sevendsen,Dingemans, Miles, GP2,yoonliu, GP1, GS, Joh2, SGP}. One of the aims of this article 
is to justify the derivation of this Korteweg-de Vries equation with topography (called KDV-top). Another development of models
 for water-waves was initiated in order to gain insight into wave 
breaking, one of the most fundamental aspects of water-waves \cite{DJ}. 
In 2008 Constantin and Lannes \cite{AD:08} 
rigorously justified the relevance of more nonlinear generalization of the 
KDV equations (linked to the Camassa-Holm equation \cite{CH} and the Degasperis-Procesi equations \cite{DP}) 
 as models for the propogation of shallow water-waves. They proved that these equations can be used to furnish approximations 
 to the governing equations for water-waves, and in their investigation they put earlier (formal) asymptotic procedures due to Johnson \cite{Jo} on a 
firm and mathematically rigorous basis. However, all these results hold for flat bottoms only. 
The main goal of this article is to investigate the same scaling as in \cite{AD:08} and to include topographical 
effects. To this end, we derive a new variable coefficients class of equations which takes into account these effects and generalizes the CH like equations of Constantin-Lannes \cite{AD:08}. 
The presence of the topography terms induce secular growth effects which do not  always allow a full justification of the model. We however  give
some consistency results for all the models derived here, and then
show that under some additional assumptions on the topography variations, the secular terms can be controled 
and a full justification given.
\subsection{Presentation of the results}\label{results}
Parameterizing the  free surface by $z=\zeta(t,x)$  (with $x\in \R$)
and the bottom by $z=-h_0+b(x)$ (with $h_0 > 0$ constant), one can use the incompressibility and
irrotationality  conditions to write  the water-waves  equations under
Bernoulli's formulation,  in terms of a  velocity potential $\varphi$
associated to the flow, and where $\varphi(t, .)$ is defined on $\Omega_t
= \{(x,z), -h_0  +b(x) < z < \zeta(t,x)\}$ (i.e. the velocity
field is given by $v=\nabla_{x,z}\varphi$) :
\begin{equation}
\left\{
\begin{array}{lcl}
\dsp\partial_x^2 \varphi + \partial_z^2 \varphi = 0, & \hbox{in} & \Omega_t ,\vspace{1mm}\\
\dsp\partial_n \varphi = 0, & \hbox{at} & z = -h_0 + b, \vspace{1mm}\\
\dsp\partial_t \zeta + \partial_x \zeta\partial_x \varphi = \partial_z \varphi, & \hbox{at} & z= \zeta, \vspace{1mm}\\
\dsp \partial_t \varphi + \frac{1}{2}((\partial_x \varphi)^2 + (\partial_z \varphi)^2) + g\zeta= 0 & \hbox{at} & z= \zeta,
\end{array}
\right.
\label{ww}
\end{equation}
where $g$  is the gravitational acceleration,  $\partial_n \varphi$ is
the outward normal  derivative at the boundary of  the fluid domain. 
The qualitative study  of the water-waves equations is  made easier by
the  introduction  of   dimensionless  variables  and  unknowns.  This
requires the introduction of various orders of magnitude linked to the
physical regime under consideration.  More precisely, let us introduce
the following quantities: $a$ is the order of amplitude of the waves;
$\lambda$ is the wave-length of the waves; $b_0$
is the order of amplitude of the variations of the bottom topography; 
$\lambda/\alpha$ is the wave-length of the bottom variations; 
$h_0$ is the reference  depth. We also introduce the following dimensionless
parameters:
\begin{equation*}
\varepsilon=\frac{a}{h_0},\quad \mu=\frac{h_0^2}{\lambda^2}, \quad 
\beta=\frac{b_0}{h_0};
\end{equation*}
the  parameter $\varepsilon$ is  often called  nonlinearity parameter;
while  $\mu$   is  the  shallowness  parameter.   We now perform the classical shallow water 
non-dimensionalization using the following relations:
\begin{equation}
\label{dim}
\begin{array}{lll}
x=\lambda x',  & z=h_0z', & \zeta=a\zeta', \\ [0.3cm]
\varphi=\displaystyle\frac{a}{h_0}\lambda\sqrt{gh_0}\varphi',& b=b_0b', & t=\frac{\lambda}{\sqrt{gh_0}}t';
\end{array}
\end{equation}
so, the equations  of motion (\ref{ww}) then become (after dropping  the
primes for the sake of clarity):
\smallskip
\begin{equation}
\left\{
\begin{array}{lcl}
\dsp\mu\partial_x^2 \varphi + \partial_z^2 \varphi = 0, & \hbox{at} & -1+\beta b^{(\alpha)}< z<\varepsilon \zeta ,\vspace{1mm}\\
\dsp\partial_z\varphi-\mu\beta\alpha\partial_xb^{(\alpha)}\partial_x\varphi=0& \hbox{at} & z = -1 + \beta b^{(\alpha)},\vspace{1mm}\\
\dsp\partial_t \zeta -\frac{1}{\mu}(\mu\varepsilon\partial_x \zeta\partial_x \varphi + \partial_z \varphi)=0, & \hbox{at} & z=\varepsilon\zeta,\vspace{1mm}\\
\dsp\partial_t \varphi + \frac{1}{2}(\varepsilon(\partial_x \varphi)^2 + \frac{\varepsilon}{\mu}(\partial_z \varphi)^2) + \zeta= 0 & \hbox{at} & z=\varepsilon \zeta,
\end{array}
\right.
\label{wwadim}
\end{equation}
where $b^{(\alpha)}(x)=b(\alpha x)$. 
Making assumptions on the respective  size of $\varepsilon$, $\beta$, $\alpha$, and
$\mu$ one is led to derive (simpler) asymptotic models from (\ref{wwadim}). In
the  shallow-water  scaling  $(\mu\ll1)$,  one  can  derive  (when no smallness
assumption is made on  $\varepsilon$, $\beta$ and $\alpha$) the so-called Green-Naghdi equations (see \cite{GN,LB} for a derivation and \cite{AL} for a
rigorous justification). For one dimensional surfaces and over uneven
bottoms these  equations couple the free surface  elevation $\zeta$ to
the vertically averaged horizontal component of the velocity,
\begin{equation}
u(t,x)=\frac{1}{1+\varepsilon\zeta-\beta b^{(\alpha)}}\int_{-1+\beta b^{(\alpha)}}^{\varepsilon\zeta}\partial_x\varphi(t,x,z)dz
\label{averaged}
\end{equation}
and can be written as:
\begin{equation}
\left\{
\begin{array}{lc}
\dsp\partial_t\zeta+\partial_x(hu)=0,\vspace{1mm}\\
\dsp(1+\frac{\mu}{h}\mathcal{T}[h,\beta b^{(\alpha)}])\partial_t u+\partial_x\zeta+\varepsilon u\partial_xu\vspace{1mm}\\
\dsp\indent+\mu\varepsilon\big\lbrace-\frac{1}{3h}\partial_x(h^3(u\partial_x^2u)-(\partial_xu)^2)+ \Im[h,\beta b^{(\alpha)}]u\big\rbrace=0
\end{array}
\right.
\label{GN1}
\end{equation}
where $h=1+\varepsilon\zeta-\beta b^{(\alpha)}$ and
\begin{eqnarray*}
\mathcal{T}[h,\beta b^{(\alpha)}]W
=-\frac{1}{3}\partial_x(h^3\partial_xW)+\frac{\beta}{2}\partial_x(h^2\partial_x b^{(\alpha)} )W+\beta^2h(\partial_xb^{(\alpha)})^2W,
\end{eqnarray*}
while the purely topographical term $\Im[h,\beta b^{(\alpha)}]u$ is defined as:
\begin{eqnarray*}
\Im[h,\beta b^{(\alpha)}]u&=&\frac{\beta}{2h}[\partial_x(h^2(u\partial_x)^2b^{(\alpha)})-h^2((u\partial_x^2u)-(\partial_xu)^2)\partial_xb^{(\alpha)})]\\&&+\beta^2((u\partial_x)^2b^{(\alpha)})\partial_xb^{(\alpha)}.
\end{eqnarray*}
If we make the additional assumption that $\varepsilon\ll1$, $\beta\ll1$ then the above system reduces at first order
to a wave equation of speed $\pm 1$ and any perturbation of the surface splits up into two components moving in opposite directions.
 A natural issue is therefore to describe more accurately the motion of these two "unidirectional" waves. In the so called long-wave regime
\begin{equation}
\mu\ll1,\quad \varepsilon=O(\mu),
\label{scalingkdv}
\end{equation}
and for flat bottoms, Korteweg and de Vries \cite{KdV} found that say, the right-going wave should satisfy the KDV equation:
\begin{equation}
u_t+u_x+\frac{3}{2}\varepsilon uu_x+\frac{\mu}{6}u_{xxx}=0,
\label{kdv}
\end{equation}
and ($\zeta=u+O(\varepsilon,\mu)$).\\
At leading order, this equation reduces to the expected transport equation at speed 1. It has been
noticed by Benjamin, Bona, Mahoney \cite{BBM} that the KDV equation belongs to a wider class of equations. For instance, the BBM equation first used by Peregrine \cite{Pe}, and sometimes also called the regularized long-wave equation, provides an approximation of the exact water-waves equations of the same accuracy as the KDV equation and can be written under the form:
 \begin{equation}
u_t+u_x+\frac{3}{2}\varepsilon uu_x+\mu(Au_{xxx}+Bu_{xxt})=0\quad\mbox{ with }\quad  A-B=\frac{1}{6}.
\label{bbm}
\end{equation}
For higher values of $\varepsilon$, the nonlinear effects are stronger; in the regime
   \begin{equation}
\mu\ll1,\quad \varepsilon=O(\sqrt{\mu}),
\label{scalingch}
\end{equation}
the BBM equations (\ref{bbm}) should be replaced by the following family (see \cite{AD:08,Jo}):
\begin{equation}
u_t+u_x+\frac{3}{2}\varepsilon uu_x+\mu(Au_{xxx}+Bu_{xxt})=\varepsilon\mu(Euu_{xxx}+Fu_xu_{xx})
\label{ch}
\end{equation}
(with some conditions on $A$, $B$, $E$, and $F$) in order to keep the same $O(\mu^2)$ accuracy of the approximation. However, all these results only hold for flat bottoms; for the situation of an uneven
bottom, various generalizations of the KDV equations with non constant coefficients
have been proposed \cite{kirby,sevendsen,Dingemans, Miles, GP2,yoonliu, GP1, GS, Joh2, SGP}. We justify in this paper the derivation of the generalized KDV equation and also
 we show that the correct generalization of the equation (\ref{ch}) under the scaling (\ref{scalingch}) and with the following conditions on the topographical variations:
 \begin{equation}
  \beta\alpha=O(\mu),\quad \beta\alpha^2=O(\mu^2) \quad \beta\alpha\varepsilon=O(\mu^{2}),
\label{scalingsd}
\end{equation}
is given by:
\begin{eqnarray}
&u_t+cu_x+\frac{3}{2}c_xu+\frac{3}{2}\varepsilon uu_x+\mu(\tilde{A}u_{xxx}+Bu_{xxt})\nonumber\\&
\quad=\varepsilon\mu\tilde{E}uu_{xxx}
+\varepsilon\mu\Big(\partial_x(\frac{\tilde{F}}{2}u)u_{xx}
+u_x\partial^2_x(\frac{\tilde{F}}{2}u)\Big)
\label{chgeneral}
\end{eqnarray}
where $c=\sqrt{1-\beta b^{(\alpha)}}$ and 
$\tilde{A}$, $\tilde{E}$, $\tilde{F}$ differ from the coefficients $A$, $E$, $F$ in (\ref{ch}) because of
topographic effects:
\begin{eqnarray*}
\tilde{A}&=&Ac^5-Bc^5+Bc \\
\tilde{E}&=&Ec^4-\frac{3}{2}Bc^4+\frac{3}{2}B
 \\\tilde{F}&=&Fc^4-\frac{9}{2}Bc^4+\frac{9}{2}B.
\end{eqnarray*}
Notice that for an equation of the family  (\ref{chgeneral}) to be linearly well-posed it is necessary that
$B\leq0$.
In Sec. \ref{sectdeux}, we derive  asymptotical approximations of  the Green-Naghdi equations over non flat bottoms:
equations on the velocity  are given in Sect. \ref{subsectdeuxun} and equations on  the surface elevation are obtained in Sect. \ref{subsectdeuxdeux}; for these equations, $L^\infty$-consistency results are given (see Definition \ref{defi}). In Sect. \ref{subsectdeuxtrois}, the same kind of result is given in the (more restrictive) KdV scaling in order to recover the many variable depth KdV equations formally derived by oceanographers.  
Section \ref{sectrois} is devoted to the study of  the well posedness  of the equations derived in Section \ref{sectdeux}. Two different approaches are used, depending on the coefficient $B$ in (\ref{chgeneral}): 
\S \ref{subsectroisun} deals with the case $B<0$ (in that case, further investigation on the breaking of waves can be performed, see \S  \ref{subsectroisdeux}) and \S  \ref{subsectroistrois} treats the case $B=0$. 
While secular growth effects prevent is from proving $H^s$-consistency (see Definition \ref{defiH^s})
for the models derived in Section  \ref{sectdeux}, we show in Section  \ref{sectquatre} that such results hold if one makes stronger assumptions on the parameters. A full justification of the models can then be given (see Th. \ref{th4}).

\section{Unidirectional limit of the Green-Naghdi equations over uneven bottom in the CH and KDV scalings}
\label{sectdeux}
We derive here asymptotical approximations of  the Green-Naghdi equations over non flat bottoms in the scalings (\ref{scalingsd}) and (\ref{scalingch}).
 We remark that the Green-Naghdi equations can then be simplified into\\ (denoting $h=1+\varepsilon\zeta-\beta b^{\alpha}$):
 \begin{equation}
 \left\{
 \begin{array}{lc}
 \dsp \zeta_t+[hu]_x=0\vspace{1mm}\\
\dsp  u_t+\zeta_x+\varepsilon uu_x=\frac{\mu}{3h}
 [h^3(u_{xt}+\varepsilon uu_{xx}-\varepsilon u_x^2)]_x,
 \end{array}
 \right.
 \label{GN3}
 \end{equation}
where $O(\mu ^2)$ terms have been discarded.\\

We consider here parameters $\varepsilon$, $\beta$, $\alpha$ and $\mu$ linked by the relations  
\begin{eqnarray}
\label{cond1}
&\varepsilon=O(\sqrt{\mu}),\;\beta\alpha=O(\varepsilon),\;
 \beta\alpha=O(\mu),\;\beta\alpha^2=O(\mu^2),\;\beta\alpha\varepsilon=O(\mu^{2})
\end{eqnarray}
(note that in the case of flat bottoms, one can take $\beta=0$, so that this set of relations reduce to $\varepsilon=O(\sqrt{\mu})$).\\
Equations for the velocity $u$ are first derived in \S 2.1 and equations for
the surface elevation $\zeta$ are obtained in \S 2.2. The considerations we make on the derivation of these equations are related to the approach initiated by Constantin and Lannes \cite{AD:08}. In addition, in \S 2.3 we recover and justify the KDV equation over a slowly varying depth (formally derived in \cite{kirby,sevendsen,Dingemans}).
\subsection{Equations on the velocity.} \label{subsectdeuxun}
If we want to find an approximation at order $O(\mu^2)$ of the GN equations under the scalings (\ref{cond1}), it is natural to look for $u$ as a solution of (\ref{chgeneral}) with variable coefficients 
 $\widetilde{A}$, $\widetilde{B}$, $\widetilde{E}$, $\widetilde{F}$ to be determined.
We prove in this section that one can associate to the solution of (\ref{chgeneral}) a family of approximate
 solutions consistent with the Green-Naghdi equations (\ref{GN3}) in the following sense:
\begin{definition}\label{defi}
 Let $\wp$ be a family of parameters $\theta=(\varepsilon,\beta,\alpha,\mu)$ satisfying (\ref{cond1}). A family $(\zeta^{\theta},u^{(\theta})_{\theta\in \wp}$ is 
$L^{\infty}$-consistent on $[0,\frac{T}{\varepsilon}]$ with the GN equations (\ref{GN3}), if for all $\theta \in \wp$ (and denoting $h^\theta=1+\varepsilon\zeta^\theta-\beta b^{(\alpha)}$),
$$\left\lbrace
    \begin{array}{l}
\dsp \zeta^{\theta}_t +[h^\theta u^\theta]_x=\mu^{2}r_{1}^{\theta}\vspace{1mm}\\
\dsp u^{\theta}_t+\zeta^{\theta}_x +\varepsilon u^{\theta}u^{\theta}_x=
\frac{\mu}{3h^\theta}
[(h^\theta)^3(u^{\theta}_{xt}+\varepsilon u^{\theta}u^{\theta}_{xx}-\varepsilon (u^{\theta}_x)^2)]_x+\mu^{2}r_{2}^{\theta}
 \end{array}
\right.$$
with $(r_{1}^{\theta},r_{2}^{\theta})_{\theta\in \wp}$ bounded in 
$L^{\infty}([0,\frac{T}{\varepsilon}]\times\R).$
\end{definition}
\begin{remark}\label{remarkconsist}
The notion of $L^{\infty}$-consistency is weaker then the notion of $H^s$-consistency
given in \S 4 (Definition \ref{defiH^s}) and does not allow a full justification of the 
asymptotic models. Since secular growth effects do not allow in general an 
$H^s$-consistency, we state here an $L^{\infty}$-consistency result under very general
assumptions on the topography parameters $\alpha$ and $\beta$. 
$H^s$-consistency and full justification of the models will then be achieved
 under additional assumptions in \S 4.
\end{remark}

The following proposition shows that there is a one parameter family of equations (\ref{chgeneral}) 
 $L^{\infty}$-consistent with the GN equations (\ref{GN3}). (For the sake of simplicity, 
here and throughout the rest of this paper, we take an infinity smooth bottom parameterized by the function $b$).
\begin{proposition}\label{propvelocity}
 Let $b\in H^{\infty}(\R)$ and $p\in \R$. Assume that 
$$A=p, \quad B=p-\frac{1}{6} \quad E=-\frac{3}{2}p-\frac{1}{6},\quad F=-\frac{9}{2}p-\frac{23}{24}.$$
Then:
\begin{itemize}
\item For all family $\wp$ of parameters satisfying (\ref{cond1}),
 \item For all $s\ge 0$ large enough and $T>0$, 
\item For all bounded family $(u^{\theta})_{\theta\in \wp} \in C([0,\frac{T}{\varepsilon}],H^{s}(\R))$ solving (\ref{chgeneral}),
\end{itemize}
the familly $(\zeta^{\theta},u^{\theta})_{\theta\in \wp}$ with (omitting the index $\theta$)\\
\begin{equation}\label{zetaexp}
 \zeta := cu+\frac{1}{2}\int_{-\infty}^{x}c_xu+\frac{\varepsilon}{4}u^{2}+\frac{\mu}{6}c^4u_{xt}-
\varepsilon\mu c^4[\frac{1}{6}uu_{xx}+\frac{5}{48}u_x^{2}],
\end{equation}
is $L^{\infty}$-consistent on $[0,\frac{T}{\varepsilon}]$ with the GN equations (\ref{GN3}).
\end{proposition}
\begin{remark}\label{remaAD}
 If we take $b=0$ -i.e if we consider a flat bottom-, then 
 one can recover the equation (7) of \cite{AD:08} and the equations (26a) and (26b) of \cite{Jo} with $p=-\frac{1}{12}$ and $p=\frac{1}{6}$ respectively.
\end{remark}
\begin{proof}
For the sake of simplicity, we denote by $O(\mu)$ 
any family of functions $(f^\theta)_{\theta\in\wp}$
such that $\displaystyle\frac{1}{\mu}f^\theta$ 
 remains bounded in $L^{\infty}([0,\frac{T}{\varepsilon}],H^{r}(\R))$
for all $\theta\in\wp$, (and for possibly different values of $r$). The same notation is also used for 
real numbers, e.g $\varepsilon=O({\mu})$, but this should not yield any confusion. We use the notation $O_{L^{\infty}}(\mu)$ if $\displaystyle\frac{1}{\mu}f^\theta$ remains bounded in $L^{\infty}([0,\frac{T}{\varepsilon}]\times\R)$.
 Of course, similar notations are used for $O(\mu^{2})$ etc. To alleviate the text,  we also
omit the index $\theta$ and write $u$ instead of $u^{\theta}$.\\
$\textbf{Step 1.}$ We begin the proof by the following Lemma where a new class of equation
is deduced from (\ref{chgeneral}). The coefficients $A$, $B$, $E$, $F$ in this new class of equations
are constants (as opposed to $\widetilde{A}$, $\widetilde{E}$ and $\widetilde{F}$ in (\ref{chgeneral}) that are functions of $x$).
\begin{lemma}\label{lemma1}
 Under the assumptions of Proposition \ref{propvelocity}, there is a family  $(R^{\theta})_{\theta\in \wp}$ bounded in 
$L^{\infty}([0,\frac{T}{\varepsilon}],H^{r}(\R))$ (for some $r<s$)
such that (omitting the index $\theta$)
\begin{eqnarray}
\label{chconsistent}
&u_t+cu_x+\frac{3}{2}c_xu+\frac{3}{2}\varepsilon u u_x +\mu c^5A u_{xxx}
 +\mu B \partial_x(c^4u_{xt}) \\&\qquad=\varepsilon\mu c^{4}E u u_{xxx}+\varepsilon\mu\frac{1}{2}F (c^{4}u)_x u_{xx}+\varepsilon\mu\frac{1}{2}Fu_x(c^{4}u)_{xx}+\mu^2R.\nonumber
\end{eqnarray}
\end{lemma}
\begin{proof}
 Remark that the relation $\alpha\beta=O(\mu)$ and the definitions of $\tilde{A}$, $\tilde{B}$, and $\tilde{E}$ in terms of $A$, $B$, $E$ and $F$ imply that 
\begin{eqnarray*}
\lefteqn{\mu(-Bc^5+Bc)u_{xxx}
+\varepsilon\mu (-\frac{3}{2}Bc^4+\frac{3}{2}B)uu_{xxx}}\nonumber\\
&&
 +\varepsilon\mu\partial_x\Big(\Big(\frac{-\frac{9}{2}Bc^4+\frac{9}{2}B}{2}\Big)u\Big)u_{xx}
+\varepsilon\mu u_x\partial^2_x\Big(\Big(\frac{-\frac{9}{2}Bc^4+\frac{9}{2}B}{2}\Big)u\Big)\nonumber\\
&&=-\mu B \partial_x(c(c^4-1)\partial_xu_x)-\frac{3}{2}\mu \varepsilon B \partial_x((c^4-1)\partial_x(uu_x))\nonumber\\
&&= -\mu B \partial_x\Big[(c^4-1)\partial_x(cu_x+\frac{3}{2}\varepsilon uu_x) \Big]+O(\mu^2)\nonumber\\
&&=\mu B \partial_x((c^4-1)u_{xt})+O(\mu^2),\nonumber
\end{eqnarray*}
the last line being a consequence of the identity $u_t=-(cu_x+\frac{3}{2}\varepsilon uu_x)+O(\mu)$ provided 
by (\ref{chgeneral}) since we have $$\vert c_xu\vert_{H^r}=\big\vert-\frac{1}{2c}\beta\alpha b^{(\alpha)}_xu\big\vert_{H^r}\leq \cst\;\alpha\beta
 \big\vert \partial_x b^{(\alpha)}\big\vert_{W^{[r]+1,\infty}} \vert u\vert_{H^r}=O(\beta\alpha)=O(\mu)$$
 where $[r]$ is the largest integer smaller or equal to $r$. 
  The equation (\ref{chgeneral}) can thus be written under the form:
\begin{eqnarray*}
&u_t+cu_x+\frac{3}{2}c_xu+\frac{3}{2}\varepsilon u u_x +\mu c^5A u_{xxx}
 +\mu B \partial_x(c^4u_{xt}) \\&\qquad=\varepsilon\mu c^{4}E u u_{xxx}+\varepsilon\mu\frac{1}{2}F (c^{4}u)_x u_{xx}+\varepsilon\mu\frac{1}{2}Fu_x(c^{4}u)_{xx}+O(\mu^2),
\end{eqnarray*}
which is exactly the result stated in the Lemma.
\end{proof} 
 If $u$ solves (\ref{chgeneral}) one also has 
\begin{eqnarray}
\label{eq1}
  u_t + cu_x +\frac{3}{2} \varepsilon uu_x &=&-\frac{3}{2}c_xu+O(\mu)\\
                                           &=&O(\mu),\nonumber
\end{eqnarray}
 Differentiating (\ref{eq1}) twice with respect to $x$, and using again the fact that $c_xf=O(\beta\alpha)=O(\mu)$ for all $f$ smooth enough,
  one gets 
\begin{equation*}
cu_{xxx}=-u_{xxt}-\frac{3}{2} \varepsilon \partial_x^2(uu_x)+O(\mu).
\end{equation*}
It is then easy to deduce that 
\begin{eqnarray*}
c^5u_{xxx}&=&-c^4u_{xxt}-\frac{3}{2} c^4\varepsilon\partial_x^2(uu_x)+O(\mu)\\
          &=&-\partial_x(c^4u_{xt})-\frac{3}{2}\varepsilon c^4(uu_{xxx}+3u_xu_{xx})+O(\mu)
\end{eqnarray*}
so that we can replace the $c^5u_{xxx}$ term of (\ref{chconsistent}) by this expression. By using Lemma \ref{lemma1}, one gets therefore the following equation where the linear term in $u_{xxx}$ has
been removed:
\begin{equation}
  u_t +cu_x+\frac{3}{2}c_xu+\frac{3}{2} \varepsilon uu_x + \mu c^4au_{xxt} =\varepsilon\mu c^4[euu_{xx}+du_x^2]_x
+O(\mu^{2})
\label{eq2}
\end{equation}
with $a=B-A$, $e=E+\frac{3}{2}A$, $d=\frac{1}{2}(F+3A-E)$.\\
$\textbf{Step 2.}$ We seek $v$ such that if $\zeta :=cu+\varepsilon v$ and $u$ solves (\ref{chgeneral}) then the second equation of (\ref{GN3}) is satisfied up to a $O(\mu^{2})$ term. This is equivalent to checking that 
\begin{eqnarray*}
	u_t+[cu+\varepsilon v]_x+\varepsilon uu_x
	&=&
        \mu(1-\beta b^{(\alpha)})\varepsilon\partial_x(cu)u_{xt}+\frac{\mu}{3}((1-\beta b^{(\alpha)})+\varepsilon cu)^2u_{xxt}\\
        &&+\frac{\varepsilon\mu}{3}(1-\beta b^{(\alpha)})^2(uu_{xx}-u_x^2)_x+O(\mu^2)\\
        &=&
        \frac{\mu}{3}c^4u_{xxt}+\varepsilon\mu \Big(c^3u_xu_{xt}+\frac{2}{3}c^3uu_{xxt}+\frac{c^4}{3}(uu_{xx}-u_x^2)_x\Big)\\
        &&+O(\mu^2),
\end{eqnarray*}
where we used the relations $O(\varepsilon^2)=O(\mu)$, $O(\beta\alpha)=O(\mu)$ and the fact that $c^2=1-\beta b^{(\alpha)}$. This condition can be recast under the form 
\begin{eqnarray*}
 &&\varepsilon v_x +[u_t +cu_x+\frac{3}{2}c_xu+\frac{3}{2} \varepsilon uu_x + \mu c^4 au_{xxt}-\varepsilon\mu c^4[euu_{xx}+du_x^2]_x]\\
&=&\frac{1}{2}c_xu+\frac{\varepsilon}{2}uu_x+\mu c^4(a+\frac{1}{3})u_{xxt}\\&&+
\varepsilon\mu c^3\Big(u_xu_{xt}+\frac{2}{3}uu_{xxt}+c[(\frac{1}{3}-e)uu_{xx}-(\frac{1}{3}+d)u_x^2]_x\Big)+O(\mu^2).
\end{eqnarray*}
Since morever one gets from (\ref{eq1}) that $u_{xt}=-cu_{xx}+O(\mu,\varepsilon)$ and $u_{xxt}=-cu_{xxx}+O(\mu,\varepsilon)$, one gets readily 
\begin{eqnarray*}
 &&\varepsilon v_x +[u_t +cu_x+\frac{3}{2}c_xu+\frac{3}{2} \varepsilon uu_x + \mu c^4 au_{xxt}-\varepsilon\mu[euu_{xx}+du_x^2]_x]\\
&&=\frac{1}{2}c_xu+\frac{\varepsilon}{2}uu_x+\mu c^4(a+\frac{1}{3})u_{xxt}-\varepsilon\mu c^4[(e+\frac{1}{3})uu_{xx}+
(d+\frac{1}{2})u_x^2]_x+O(\mu^{2}).
\end{eqnarray*}
From Step 1, we know that the term between brackets in the lhs of this equation is of order
$O(\mu^2)$ so that the second equation of (\ref{GN3}) is satisfied up to $O(\mu^2)$ terms if 
\begin{equation}\label{eqon v}
 \varepsilon v_x=\frac{1}{2}c_xu+\frac{\varepsilon}{2}uu_x+\mu c^4(a+\frac{1}{3})u_{xxt}-\varepsilon\mu c^4[(e+\frac{1}{3})uu_{xx}+(d+\frac{1}{2})u_x^2]_x +O(\mu^{2}).
\end{equation}
At this point we need also the following lemma
\begin{lemma}\label{lemma2}
With $u$ and $b$ as in the statement of  Proposition \ref{propvelocity}, 
the mapping $(t,x)\longrightarrow\int_{-\infty}^{x}c_xu\; dx $ is well defined on $[0,\frac{T}{\varepsilon}]\times\R$.\\ Morever one has that
$$
\Big\vert\int_{-\infty}^{x}c_xu \;dx\Big\vert_{L^{\infty}([0,\frac{T}{\varepsilon}]\times\R)}\leq \cst \sqrt{\alpha}
\beta \vert b_x\vert_2 \vert u\vert_2.
$$
\end{lemma}
\begin{proof}
We used here the Cauchy-Schwarz inequality and the definition $c^2=1-\beta b^{(\alpha)}$ to get
$$
\int_{-\infty}^{x}c_xu \;dx\leq \cst\;\alpha\beta\vert (b_x)^{(\alpha)} \vert_{2}\vert u \vert_2 \leq\cst\;
\sqrt{\alpha}\beta\vert b_x \vert_{2}\vert u \vert_2< \infty.
$$
It is then easy to conclude the proof of the lemma.
\end{proof}
Thanks to this lemma there is a solution $v\in C([0,\frac{T}{\varepsilon}]\times\R)$ to (\ref{eqon v}), namely
\begin{equation}
 \varepsilon v=\frac{1}{2}\int_{-\infty}^{x}c_xu+\frac{\varepsilon}{4}u^{2}+\mu c^4(a+\frac{1}{3})u_{xt}-\varepsilon\mu c^4[(e+\frac{1}{3})uu_{xx}+
(d+\frac{1}{2})u_x^2].
\label{eq3}
\end{equation}
$\textbf{Step 3.}$ We show here that it is possible to choose the coefficients $A$, $B$, $E$, $F$ such that the first equation of (\ref{GN3}) is also satisfied up to $O_{L^\infty}(\mu^{2})$ terms.
This is equivalent to checking that
\begin{equation}
[cu+\varepsilon v]_t +[(1+\varepsilon(cu+\varepsilon v)-\beta b^{(\alpha)})u]_x =0
\label{eq4}
\end{equation}
Remarking that the relations $O(\beta\alpha)=O(\mu)$, $O(\beta\alpha^2)=O(\mu^2)$, and $O(\beta\alpha\varepsilon)=O(\mu^2)$ imply that
$$\frac{1}{2}\int_{-\infty}^{x}c_xu_t=-\frac{1}{2}cc_xu+O_{L^{\infty}}(\mu^2),$$
one infers from (\ref{eq3}) that
\begin{eqnarray*}
\varepsilon v_t&=&\frac{1}{2}\int_{-\infty}^{x}c_xu_t+\frac{\varepsilon}{2}uu_t +\mu c^4(a+\frac{1}{3})u_{xtt}-\varepsilon\mu c^4[(e+\frac{1}{3})uu_{xx}+
(d+\frac{1}{2})u_x^2]_t\\
&=& -\frac{1}{2}cc_xu-\frac{\varepsilon}{2}u(cu_x+\frac{3\varepsilon}{2}uu_x+\mu c^4au_{xxt})
-\mu c^4(a+\frac{1}{3})
\partial_{xt}^2(cu_x+\varepsilon\frac{3}{2}uu_x)
\\&&+\varepsilon\mu c^5[(e+\frac{1}{3})uu_{xx}+
(d+\frac{1}{2})u_x^2]_x+O(\mu^{2})+O_{L^{\infty}}(\mu^2)\\
&=&-\frac{1}{2}cc_xu-\varepsilon\frac{1}{2}cuu_x-\varepsilon^2\frac{3}{4}u^2u_x -\mu(a+\frac{1}{3})c^5u_{xxt}
\\&&+ \varepsilon\mu c^5[(2a+e+\frac{5}{6})uu_{xx}+(\frac{5}{4}a+d+1)u_x^2]_x+O(\mu^{2})+O_{L^{\infty}}(\mu^2).
\end{eqnarray*}
Similarly, one gets 
\begin{equation*}
 \varepsilon^2[vu]_x=\varepsilon^2\frac{3}{4}u^2u_x -\varepsilon\mu c^5(a+\frac{1}{3})[uu_{xx}]_x + O(\mu^{2})+O_{L^{\infty}}(\mu^2),
\end{equation*}
 so (\ref{eq4}) is equivalent to 
\begin{eqnarray*}
 cu_t+\varepsilon v_t+c^2u_x+2cc_xu+2\varepsilon cuu_x +\varepsilon^2[vu]_x =O_{L^{\infty}}(\mu^2).
\end{eqnarray*}
Multiplying by $\displaystyle\frac{1}{c}$, we get
\begin{eqnarray*}
 &&u_t+cu_x+\frac{3}{2}c_xu +\varepsilon\frac{3}{2}uu_x-\mu c^4(a+\frac{1}{3})u_{xxt}\\
&=&\varepsilon\mu c^4[-(e+a+\frac{1}{2})uu_{xx}-(\frac{5}{4}a+d+1)u_x^2]_x+O_{L^{\infty}}(\mu^{2}). 
\end{eqnarray*}
Equating the coefficients of this equation with those of (\ref{eq2}) shows that the first equation of (\ref{GN3})
is also satisfied at order $O_{L^\infty}(\mu^2)$ if the following relations hold:
\begin{equation*}
 a=-\frac{1}{6},\quad e=-\frac{1}{6}, \quad d=-\frac{19}{48},
\end{equation*}
and the conditions given in the statement of the proposition on $A$, $B$, $E$, and $F$ follows from the expressions of $a$, $e$ and $d$  given after (\ref{eq2}).
\end{proof}
\subsection{Equations on the surface elevation.} \label{subsectdeuxdeux}
Proceeding exactly as in the proof of Proposition \ref{propvelocity}, one can prove that the family of equations on the surface elevation
\begin{eqnarray}\label{chgeneralzeta}
 \lefteqn{\zeta_t+c\zeta_x+\frac{1}{2}c_x\zeta+\frac{3}{2c} \varepsilon\zeta\zeta_x
-\frac{3}{8c^3}\varepsilon^2\zeta^2\zeta_x+
\frac{3}{16c^5}\varepsilon^3\zeta^3\zeta_x}\nonumber\\&&
+\mu(\tilde{A}\zeta_{xxx}+ B\zeta_{xxt})=
\varepsilon\mu \tilde{E}\zeta\zeta_{xxx}
+\varepsilon\mu\Big(\partial_x(\frac{\tilde{F}}{2}\zeta)\zeta_{xx}
+ \zeta_x\partial^2_x(\frac{\tilde{F}}{2}\zeta)\Big),
\end{eqnarray}
where
\begin{eqnarray*}
 \tilde{A}&=&Ac^5-Bc^5+Bc \\
\tilde{E}&=&Ec^3-\frac{3}{2}Bc^3+\frac{3}{2c}B
 \\\tilde{F}&=&Fc^3-\frac{9}{2}Bc^3+\frac{9}{2c}B,
\end{eqnarray*}
 can be used to construct an approximate solution 
consistent with the Green-Naghdi equations:
\begin{proposition}\label{propzeta}
  Let $b\in H^{\infty}(\R)$ and $q\in \R$. Assume that 
$$A=q, \quad B=q-\frac{1}{6} \quad E=-\frac{3}{2}q-\frac{1}{6},\quad F=-\frac{9}{2}q-\frac{5}{24}.$$
Then:
\begin{itemize}
\item For all family $\wp$ of parameters satisfying (\ref{cond1}),
 \item For all $s\ge 0$ large enough  and $T>0$, 
\item For all bounded family $(\zeta^{\theta})_{\theta\in \wp} \in C([0,\frac{T}{\varepsilon}],H^{s}(\R))$ solving (\ref{chgeneralzeta}),
\end{itemize}
the familly $(\zeta^{\theta},u^{\theta})_{\theta\in \wp}$ with (omitting the index $\theta$)
\begin{eqnarray}
 \lefteqn{u:= \frac{1}{c}\Big(\zeta+\frac{c^2}{c^2+\varepsilon\zeta}\Big(-\frac{1}{2}\int_{-\infty}^{x}\frac{c_x}{c}\zeta
        -\frac{\varepsilon}{4c^2}\zeta^2
	-\frac{\varepsilon^2}{8c^4}\zeta^3
	+\frac{3\varepsilon^3}{64c^6}\zeta^4}\label{exp u}\\&&
	\qquad-\mu\frac{1}{6}c^3\zeta_{xt}
	+\varepsilon\mu c^2\big[\frac{1}{6}\zeta\zeta_{xx}+\frac{1}{48}
	\zeta_x^2\big]\Big)\Big)\nonumber
\end{eqnarray}
is $L^\infty$-consistent  on $[0,\frac{T}{\varepsilon}]$ with the GN equations (\ref{GN3}).
\end{proposition}
\begin{remark}\label{remaAD1}
 If we take $b=0$ -i.e if we consider a flat bottom-, then 
 one can recover the equation (18) of \cite{AD:08}.
\end{remark}
\begin{remark}\label{remabw}
 Choosing $q=\frac{1}{12}$, $\alpha=\varepsilon$ and $\beta=\mu^{3/2}$ the equation (\ref{chgeneralzeta}) reads
after neglecting the $O(\mu^2)$ terms: 
\begin{eqnarray}
 \lefteqn{\zeta_t+c\zeta_x+\frac{1}{2}c_x\zeta+\frac{3}{2} \varepsilon\zeta\zeta_x
-\frac{3}{8}\varepsilon^2\zeta^2\zeta_x+
\frac{3}{16}\varepsilon^3\zeta^3\zeta_x}\nonumber\\&&
+\frac{\mu}{12}( \zeta_{xxx}-\zeta_{xxt})=-\frac{7}{24}
\varepsilon\mu(\zeta\zeta_{xxx}+2\zeta_{x}\zeta_{xx}),\label{chgbw}
\end{eqnarray}
it is more advantageous to use this equation (\ref{chgbw})  to study the pattern of  wave-breaking 
for the  variable bottom CH equation  (see  \S 3.2 below).
\end{remark}
\begin{proof}
As in Proposition  \ref{propvelocity}, we make the proof in 3 steps (we just sketch the proof here since it is similar to the proof of Proposition 
\ref{propvelocity}). \\
$\textbf{Step 1.}$ We  prove that if $\zeta$ solves (\ref{chgeneralzeta}) one gets 
\begin{eqnarray}
  &\zeta_t +c\zeta_x+\frac{1}{2}c_xu+\frac{3}{2c} \varepsilon \zeta\zeta_x
-\frac{3}{8c^3}\varepsilon^2\zeta^2\zeta_x
+\frac{3}{16c^5}\varepsilon^3\zeta^3\zeta_x
 +\mu c^4a\zeta_{xxt} \label{eq1'}\\
&\qquad\qquad=\varepsilon\mu c^3[e\zeta\zeta_{xx}+d\zeta_x^2]_x
+O(\mu^{2})\nonumber
\end{eqnarray}
with $a=B-A$, $e=E+\frac{3}{2}A$, $d=\frac{1}{2}(F+3A-E)$.\\
$\textbf{Step 2.}$ We seek $v$ such that if $u:=\displaystyle\frac{1}{c}(\zeta+\varepsilon v)$ and $\zeta$ solves (\ref{chgeneralzeta}) then the first equation of (\ref{GN3}) is satisfied up to a
 $O_{L^\infty}(\mu^{2})$ term. Proceeding as in the proof of Proposition \ref{propvelocity},
one can check that a good choice for $v$ is
\begin{eqnarray}
&&\Big(\frac{c^2+\varepsilon\zeta}{c^2}\Big)\varepsilon v=-\frac{1}{2}\int_{-\infty}^{x}\frac{c_x}{c}\zeta
-\frac{\varepsilon}{4c^2}\zeta^2-\frac{\varepsilon^2}{8c^4}\zeta^3+\frac{3\varepsilon^3}{64c^6}\zeta^4
\label{exp of v zeta}\\&&
\qquad\qquad+\mu c^3a\zeta_{xt}-\varepsilon\mu c^2[e\zeta\zeta_{xx}+
d\zeta_x^2].\nonumber
\end{eqnarray}
$\textbf{Step 3.}$ We show here that it is possible to choose the coefficients $A$, $B$, $E$, $F$ such that the second equation of (\ref{GN3}) is also satisfied up to $O_{L^{\infty}}(\mu^{2})$ terms.
Replacing $u$ by $\displaystyle\frac{1}{c}(\zeta+\varepsilon v)$ with $v$ given by (\ref{exp of v zeta}), one can check that
that this condition is equivalent to
\begin{eqnarray*}
 &&\zeta_t+c\zeta_x+\frac{1}{2}c_x\zeta +\frac{3}{2c}\varepsilon\zeta\zeta_x
-\frac{3}{8c^3}\varepsilon^2\zeta^2\zeta_x
+\frac{3}{16c^5}\varepsilon^3\zeta^3\zeta_x-\mu c^4(a+\frac{1}{3})\zeta_{xxt}\\
&=&\varepsilon\mu c^3[-(e+a-\frac{1}{6})\zeta\zeta_{xx}-(\frac{7}{4}a+d+\frac{1}{3})\zeta_x^2]_x+O_{L^\infty}(\mu^{2}). 
\end{eqnarray*}
Equating the coefficients of this equation with those of (\ref{eq1'}) shows that the second equation of (\ref{GN3})
is also satisfied at order $O_{L^\infty}(\mu^2)$ if the following relations hold:
\begin{equation*}
 a=-\frac{1}{6},\quad e=-\frac{1}{6}, \quad d=-\frac{1}{48},
\end{equation*}
and the conditions given in the statement of the proposition on $A$, $B$, $E$, and $F$ follows from the expressions of $a$, $e$ and $d$  given after (\ref{eq1'}).
\end{proof}
\subsection{Derivation of the KdV equation in the long-wave scaling.} \label{subsectdeuxtrois}
 In this subsection, attention is given to the regime of slow variations of the bottom topography under the long-wave scaling $\varepsilon=O(\mu)$. We give here a rigorous justification in the meaning of consistency of the variable-depth extensions of the KdV equation (called KdV-top) originally derived in \cite{kirby,sevendsen,Dingemans}. We consider the values of $\varepsilon$, $\beta$, $\alpha$ and $\mu$ satisfying:
\begin{equation}
 \varepsilon=O(\mu),\quad \alpha\beta=O(\varepsilon),\quad \alpha^2\beta=O(\varepsilon^2).
\label{cond2}
\end{equation}
\begin{remark}
\label{remarkwp}
Any family of parameters $\theta=(\varepsilon,\beta,\alpha,\mu)$ satisfying (\ref{cond2}),  also  satisfies 
(\ref{cond1}) .
\end{remark}
 Neglecting the $O(\mu^2)$ terms, one obtains from (\ref{GN3}) the following Boussinesq system: 
\begin{equation}
 \left\{
 \begin{array}{lc}
 \dsp \zeta_t+[h u]_x=0\vspace{1mm}\\
 \dsp u_t+\zeta_x+\varepsilon uu_x=\frac{\mu}{3}c^4u_{xxt},
 \end{array}
 \right.
 \label{BOUSS}
 \end{equation}
where we recall that $h=1+\varepsilon\zeta-\beta b^{(\alpha)}$ and $c^2=1-\beta b^{(\alpha)}$.
The next proposition proves that the KdV-top equation 
\begin{equation}
 \zeta_t + c\zeta_x +\frac{3}{2c} \varepsilon \zeta\zeta_x + 
\frac{1}{6}\mu c^5 \zeta_{xxx} +\frac{1}{2}c_x\zeta=0,
\label{kdv-top}
\end{equation}
 is $L^\infty$-consistent  with the equations (\ref{BOUSS}).
\begin{proposition}\label{propbouss}
 Let $b\in H^{\infty}(\R)$.
 Then:
\begin{itemize}
\item For all family $\wp'$ of parameters satisfying (\ref{cond2}),
 \item For all $s\ge 0$ large enough and $T>0$, 
\item For all bounded family $(\zeta^{\theta})_{\theta\in \wp'} \in C([0,\frac{T}{\varepsilon}],H^{s}(\R))$ solving (\ref{kdv-top})
\end{itemize}
the familly $(\zeta^{\theta},u^{\theta})_{\theta\in \wp'}$ with (omitting the index $\theta$)\\
$$u:= \frac{1}{c}\Big(\zeta-\frac{1}{2}\int_{-\infty}^{x}\frac{c_x}{c}\zeta-\frac{\varepsilon}{4c^2}\zeta^2
	+\mu\frac{1}{6}c^4\zeta_{xx}\Big)$$
is $L^\infty$-consistent  on $[0,\frac{T}{\varepsilon}]$ with the equations (\ref{BOUSS}).
\end{proposition}
\begin{remark}\label{remarkV}
Similarly, one can prove that a family $(\zeta^\theta,u^\theta)_{\theta\in\wp'}$
with $u^\theta$ solution of the  KDV-top equation
\begin{equation}\label{kdv-topV}
 u_t + cu_x +\frac{3}{2}\varepsilon uu_x + 
\frac{1}{6}\mu c^5 u_{xxx} +\frac{3}{2}c_xu=0,
\end{equation}
and $\zeta^\theta$ given by
\begin{equation}\label{zetaexpkdv}
 \zeta := cu+\frac{1}{2}\int_{-\infty}^{x}c_xu+\frac{\varepsilon}{4}u^{2}-\frac{\mu}{6}c^5u_{xx},
\end{equation}
 is $L^\infty$-consistent  with the equations (\ref{BOUSS}).
\end{remark}
\begin{proof}
 We saw in the previous subsection that if $(\zeta^{\theta})_{\theta\in\wp}$ is a family of solutions of (\ref{chgeneralzeta}), then the family $(\zeta^{\theta},u^{\theta})_{\theta\in\wp}$ with $u^{\theta}$
is given by (\ref{exp u}) is $L^\infty$-consistent with the equations (\ref{GN3}).
Since $\theta=(\varepsilon,\beta,\alpha,\mu)\in\wp'\subset\wp$ then
by taking $q=\frac{1}{6}$, we remark that the equations (\ref{chgeneralzeta}) and (\ref{kdv-top})
are equivalent in the meaning of $L^\infty$-consistency
and the systems (\ref{GN3}) and (\ref{BOUSS})
are also, so it is clear that $(\zeta^{\theta},u^{\theta})_{\theta\in\wp'}$
is $L^\infty$-consistent  with the equations (\ref{BOUSS}).
\end{proof}
\section{Mathematical analysis of the variable bottom models} \label{sectrois}
\subsection{Well-posedness for the variable bottom CH equation} \label{subsectroisun}
We prove here the well posedness of the general class of equations 
\begin{eqnarray}
\label{wpchgeneral}
&\dsp (1-\mu m\partial_x^2)u_t+cu_x+kc_xu+\sum_{j\in J}\varepsilon^jf_j u^j u_x+\mu gu_{xxx}\\
&\dsp \qquad\qquad\qquad=\varepsilon\mu\Big[h_1u u_{xxx}+\partial_x(h_2u) u_{xx}+ u_x\partial_x^2(h_2u)\Big],\nonumber
\end{eqnarray}
where $m>0$ ,  $k\in\R$ , $J$ is a finite subset of $\N^*$ and $f_j=f_j(c)$, $g=g(c)$, $h_1=h_1(c)$ and $h_2=h_2(c)$ are smooth functions of $c$.
We also recall  that $c=\sqrt{1-\beta b^{(\alpha)}}$.
\begin{example}
Taking
\begin{eqnarray*}
 &&m=-B,\qquad k=\frac{3}{2},\qquad
J=\{1\},\qquad f_1(c)=\frac{3}{2},\\
&&g(c)=Ac^5-Bc^5+Bc, \qquad h_1(c)=c^4E-\frac{3}{2}Bc^4+\frac{3}{2}B,\\
&&h_2(c)=\frac{Fc^4-\frac{9}{2}Bc^4+\frac{9}{2}B}{2},
\end{eqnarray*}
the equation (\ref{wpchgeneral}) coincides with (\ref{chgeneral}).
\end{example}
\begin{example} Taking
\begin{eqnarray*}
&&m=-B,\qquad
k=\frac{1}{2},\qquad J=\{1,2,3\},\\
&&f_1(c)=\frac{3}{2c},\qquad
f_2(c)=-\frac{3}{8c^3},\qquad
f_3(c)=\frac{3}{16c^5}, \\ 
&&g(c)=Ac^5-Bc^5+Bc,\qquad
h_1 (c)=c^3E-\frac{3}{2}Bc^3+\frac{3}{2c}B,\\
&&h_2(c)=\frac{Fc^3-\frac{9}{2}Bc^3+\frac{9}{2c}B}{2},
\end{eqnarray*}
the equation (\ref{wpchgeneral}) coincides with (\ref{chgeneralzeta}). 
\end{example}

More precisely,  Theorem \ref{th1} below shows that one can solve the initial
value problem
\begin{equation}\label{chgeneralsys}
	\left\vert
	\begin{array}{l}
	\dsp (1-\mu m\partial_x^2)u_t+cu_x+kc_xu+\sum_{j\in J}\varepsilon^jf_j u^j u_x+\mu gu_{xxx}
        \\
        \dsp \hskip 1.5cm =\varepsilon\mu\Big[h_1u u_{xxx}+\partial_x(h_2u) u_{xx}+ u_x\partial_x^2(h_2u)\Big],\\
	\dsp u_{\vert_{t=0}}=u^0
	\end{array}\right.
\end{equation}
on a time scale
$O(1/\varepsilon)$, and under the condition $m>0$. In order to state the
result, we need to define the energy space $X^s$ ($s\in\R$) as
$$
	X^{s+1}(\R)=H^{s+1}(\R)\mbox{ endowed with the norm }\vert f\vert_{X^{s+1}}^2
	=\vert f\vert_{H^s}^2+\mu m\vert \partial_x f\vert_{H^s}^2.
$$
\begin{theorem}\label{th1}
	Let $m>0$, $s>\frac{3}{2}$ and $b\in H^{\infty}(\R)$. Let also $\wp$  be a  family of parameters   $\theta=(\varepsilon,\beta,\alpha,\mu)$
	 satisfying (\ref{cond1}).
        Then for all $u^0 \in H^{s+1}(\R)$, there exists $T>0$ and a unique family of solutions 
	$(u^{\theta})_{\theta\in{\wp}}$ to (\ref{chgeneralsys})
	bounded in $C([0,\frac{T}{\varepsilon}];X^{s+1}(\R))\cap
	C^1([0,\frac{T}{\varepsilon}];X^{s}(\R))$.
\end{theorem}
\begin{proof}
In this proof, we use the generic notation
$$
C=C(\varepsilon,\mu,\alpha,\beta,s,\vert b\vert_{H^\sigma})
$$
for some $\sigma>s+1/2$ large enough. When the constant also depends on $\vert v\vert_{X^{s+1}}$,
we write $C(\vert v\vert_{X^{s+1}})$.
Note that the dependence on the parameters is assumed to be nondecreasing. 

For all $v$ smooth enough, let us define the ``linearized'' operator
${\mathcal L}(v,\partial)$ as
\begin{eqnarray*}
	{\mathcal L}(v,\partial)&=&(1-\mu m\partial_x^2)\partial_t
	+c\partial_x+kc_x+\varepsilon^j f_j v^j\partial_x+\mu g\partial_x^3
        \\
	&&-\varepsilon\mu\big[ h_1 v\partial_x^3+ (h_2v)_x\partial_x^2+(h_2v)_{xx}\partial_x\big];
\end{eqnarray*}
for the sake of simplicity we use the convention of summation over repated indexes,  $\sum_{j\in J}\varepsilon^jf_j v^j=\varepsilon^j f_j v^j$.
To build a solution of (\ref{wpchgeneral}) using an iterative scheme, 
we have to study the problem initial value
\begin{equation}\label{linsys}
	\left\lbrace
	\begin{array}{l}
	{\mathcal L}(v,\partial)u=\varepsilon f,\\
	u_{\vert_{t=0}}=u^0.
	\end{array}\right.
\end{equation}
If $v$ is smooth enough, 
it is completely standard to check that for all $s\geq 0$,
$f \in L^1_{loc}(\R^+_t;H^s(\R_x))$ and $u^0\in H^s(\R)$, there
exists a unique solution $u\in C(\R^+;H^{s+1}(\R))$ to (\ref{linsys})
(recall that $m>0$). We thus  take for granted the existence of a
solution to (\ref{linsys}) and establish some precise energy estimates
on the solution. These energy estimates are given in terms of the $\vert \cdot \vert_{X^{s+1}}$  norm introduced above:
$$
	\vert u\vert_{X^{s+1}}^2=\vert u\vert_{H^s}^2+\mu m\vert  \partial_x u\vert_{H^s}^2.
$$
Differentiating $\frac{1}{2}e^{-\varepsilon\lambda t}\vert u\vert_{X^{s+1}}^2$
with respect to time, one gets using the equation (\ref{linsys}) and integrating by parts,
\begin{eqnarray*}
	\frac{1}{2}e^{\varepsilon\lambda t}\partial_t (e^{-\varepsilon\lambda t}\vert u\vert_{X^{s+1}}^2)
	&=&-\frac{\varepsilon\lambda}{2} \vert u\vert_{X^{s+1}}^2
        -(\Lambda^s (c\partial_x u),\Lambda^su)-k(\Lambda^s (u\partial_xc ),\Lambda^s u)\\
	&&+ \varepsilon (\Lambda^s f,\Lambda^su)
	-\varepsilon^j (\Lambda^s(f_jv^j\partial_x u),\Lambda^su)-\mu (\Lambda^s(g\partial_x^3 u),\Lambda^s u)\\
	& &-\varepsilon\mu (\Lambda^s(h_1v\partial_x^3 u),\Lambda^s u)
	-\varepsilon\mu 
	(\Lambda^s((h_2v)_x\partial_x u),\Lambda^s \partial_x u),
\end{eqnarray*}
where $\Lambda=(1-\partial_x^2)^{1/2}$. 
Since for all constant coefficient skewsymmetric differential 
polynomial $P$ (that is, $P^*=-P$), and all $h$ smooth enough, one has
$$
	(\Lambda^s (h P u),\Lambda^s u)=
	([\Lambda^s,h]Pu,\Lambda^s u)
	-\frac{1}{2}([P,h]\Lambda ^su,\Lambda^s u),
$$
we deduce (applying this identity with $P=\partial_x$ and $P=\partial_x^3$),
\begin{eqnarray}
	\lefteqn{\frac{1}{2}e^{\varepsilon\lambda t}\partial_t(e^{-\varepsilon\lambda t} \vert u\vert_{X^{s+1}}^2)
	=-\frac{\varepsilon\lambda}{2} \vert u\vert_{X^{s+1}}^2+
        \varepsilon (\Lambda^s f,\Lambda^su)}\nonumber\\
      && -\big([\Lambda^s,c]\partial_xu,\Lambda^s u\big)+
         \frac{1}{2}((\partial_xc) \Lambda^su, \Lambda^s u\big)
        -k(\Lambda^s (u\partial_xc ),\Lambda^su)\nonumber\\
        &&-\varepsilon^j \big([\Lambda^s,f_jv^j]\partial_x u,\Lambda^s u\big)
	+\frac{\varepsilon^j}{2}(\partial_x(f_jv^j)\Lambda^s u,\Lambda^s u\big)\nonumber\\
	& &-\mu \big([\Lambda^s,g]\partial_x^2 u-\frac{3}{2}g_x\Lambda^s\partial_x u-g_{xx}\Lambda^s u,\Lambda^s \partial_x u\big)-\mu\big([\Lambda^s,g_x]\partial_x^2 u,\Lambda^s u\big)\nonumber\\&&
       -\varepsilon\mu \big([\Lambda^s,h_1v]\partial_x^2 u-\frac{3}{2}(h_1v)_x\Lambda^s\partial_x u-(h_1v)_{xx}\Lambda^s u,\Lambda^s \partial_x u\big)\nonumber\\
	&&-\varepsilon\mu \big([\Lambda^s,(h_1v)_x]\partial_x^2 u,\Lambda^s u\big)
	-\varepsilon\mu \big(\Lambda^s((h_2v)_x\partial_x u),\Lambda^s \partial_x u\big).
\label{existenceCH}
\end{eqnarray}
Note that we  also used the identities
$$
[\Lambda^s, h]\partial_x^3 u=\partial_x\big([\Lambda^s,h]\partial_x^2 u\big)-[\Lambda^s,h_x]\partial_x^2 u
$$
and 
$$
	\frac{1}{2}(h_{xxx}\Lambda^s u,\Lambda^s u)=-(h_{xx}\Lambda^s u,\Lambda^s u_x).
$$
The terms involving the velocity $c$ (second line in the r.h.s of (\ref{existenceCH})) are controled using the following lemma:
\begin{lemma}\label{lemma3}
 Let $s>3/2$  and $(\varepsilon,\beta,\alpha,\mu)$ satisfy (\ref{cond1}). Then there exists $C>0$ such that 
\begin{eqnarray}
\vert\big([\Lambda^s,c]\partial_xu,\Lambda^su\big)\vert&\leq& \varepsilon C \vert u\vert_{X^{s+1}}^2\label{estiamte3}\\
\vert \big((\partial_xc) \Lambda^su, \Lambda^s u\big)\vert&\leq& \varepsilon C \vert u\vert_{X^{s+1}}^2
\label{estiamte1}\\
\vert\big(\Lambda^s (u\partial_xc ),\Lambda^su\big)\vert&\leq& \varepsilon C \vert u\vert_{X^{s+1}}^2.
\label{estiamte2}
\end{eqnarray}
\end{lemma}
\begin{proof}
$\bullet$ Estimate of $\vert\big([\Lambda^s,c]\partial_xu,\Lambda^su\big)\vert$.
One could control this term by a standard commutator estimates in terms of $\vert c_x\vert_{H^{s-1}}$;
however, one has $\vert c_x\vert_{H^{s-1}}=O(\sqrt{\alpha}\beta)$ and not $O(\alpha\beta)$ as
needed. We thus  write 
$$
\big([\Lambda^s,c]\partial_xu,\Lambda^su\big)=\big(([\Lambda^s,c]-\{\Lambda^s,c\})\partial_xu,\Lambda^su\big)
+\big(\{\Lambda^s,c\}\partial_xu,\Lambda^su\big),
$$
where for all function $F$, $\{\Lambda^s,F\}$ stands for the Poisson bracket,
$$
	\{\Lambda^s,F\}= -s\partial_x F\Lambda^{s-2}\partial_x.
$$
We can then use the following  commutator estimate (\cite{lannes'}, Theorem 5): for all
$F$ and $U$ smooth enough, one has
$$ \forall s> 3/2,\qquad
 \vert([\Lambda^s,F]-\{\Lambda^s,F\})U\vert_2\leq \cst\; \vert \partial_x^2F\vert_{H^s} \vert U\vert_{H^{s-2}}^2.
$$ 
Since 
$ \vert \partial_x^2c\vert_{H^s}=O(\alpha\beta)=O(\varepsilon)$, we deduce
$$\vert\big(([\Lambda^s,c]-\{\Lambda^s,c\})\partial_xu,\Lambda^su\big)\vert\leq C \beta\alpha \vert u\vert_{X^{s+1}}^2
\leq C \varepsilon \vert u\vert_{X^{s+1}}^2.$$ 
Morever
$$
\vert\big(\{\Lambda^s,c\}\partial_xu,\Lambda^su\big)\vert\leq\vert -s\partial(c^4)\Lambda^{s-2}\partial_x^2u\vert_2\vert u\vert_{X^{s+1}}
\leq C \vert \partial(c^4)\vert_{\infty} \vert u\vert_{X^{s+1}}^2\leq \varepsilon C \vert u\vert_{X^{s+1}}^2,
$$
and (\ref{estiamte3}) thus follows easily.\\
$\bullet$ Estimate of $\vert\big((\partial_xc) \Lambda^su, \Lambda^s u\big)\vert$. Since $\vert u\vert_{H^s}\leq \vert u\vert_{X^{s+1}}$,
one has by Cauchy-Schwarz inequality
$$\big((\partial_xc) \Lambda^su, \Lambda^s u\big)\leq \vert \partial_x c \vert_{\infty}\vert u\vert_{X^{s+1}}^2;$$
therefore, by using the fact that $\vert \partial_xc \vert_\infty\leq\cst \;\alpha\beta\vert \partial_xb \vert_\infty$ and
$\alpha\beta=O(\varepsilon)$, one gets easily 
$$\big((\partial_xc) \Lambda^su, \Lambda^s u\big)\leq \varepsilon C \vert u\vert_{X^{s+1}}^2.$$
$\bullet$ Estimate of $\vert\big(\Lambda^s (u\partial_xc ),\Lambda^su\big)\vert$. By Cauchy-Schwarz we get
$$\vert(\Lambda^s (u\partial_xc ),\Lambda^su)\vert\leq\vert u\partial_xc \vert_{H^s} \vert u\vert_{H^s}.$$
We have also that
$$
\vert u\partial_xc\vert_{H^s} \leq \vert \partial_xc\vert_{W^{[s]+1,\infty}}\vert u\vert_{X^{s+1}}
$$
where $[s]$ is the largest integer smaller or equal to $s$ (this estimate is obvious for
$s$ integer and is obtained by interpolation for non integer values of $s$.) 
Using this estimate, and the fact that 
$
\beta\alpha=O(\varepsilon),
$
it is easy to deduce (\ref{estiamte2}).
\end{proof}
For the terms involving the $f_j$, (third line in the r.h.s of (\ref{existenceCH})) we use the controls given by:
\begin{lemma}\label{lemma3'}
 Under the assumptions of Theorem \ref{th1}, one has that
\begin{eqnarray}
\vert\varepsilon^j\big([\Lambda^s,f_jv^j]\partial_x u,\Lambda^s u\big)\vert&\leq& \varepsilon C( \vert v\vert_{X^{s+1}})\vert u\vert_{X^{s+1}}^2\label{estimate4}\\
\vert \varepsilon^j\big(\partial_x(f_jv^j)\Lambda^s u,\Lambda^s u\big)\vert
&\leq& \varepsilon C( \vert v\vert_{X^{s+1}})\vert u\vert_{X^{s+1}}^2\label{estimate5}.
\end{eqnarray}
\end{lemma}
\begin{proof}
$\bullet$ Estimate of $\varepsilon^j\big([\Lambda^s,f_jv^j]\partial_x u,\Lambda^s u\big)$. By Cauchy-Schwarz we get
$$
\vert\varepsilon^j\big([\Lambda^s,f_jv^j]\partial_x u,\Lambda^s u\big)\vert\leq\varepsilon^j\vert [\Lambda^s,f_jv^j]\partial_x u\vert_2 \vert u\vert_{X^{s+1}}.
$$
Let us recall here the well-known Calderon-Coifman-Meyer commutator estimate: for all
$F$ and $U$ smooth enough, one has
$$
	\forall s>3/2,\qquad
\vert [\Lambda^s,F]U\vert_2\leq \cst \vert F\vert_{H^s}\vert U\vert_{H^{s-1}};
$$
using this estimate,  it is easy to check that one gets (\ref{estimate4}).\\
$\bullet$ Estimate of $\varepsilon^j\big(\partial_x(f_jv^j)\Lambda^s u,\Lambda^s u\big)$. It is clear that 
$$
\vert\big(\partial_x(f_jv^j)\Lambda^s u,\Lambda^s u\big)\vert\leq\vert\partial_x(f_jv^j) \vert_{\infty}\vert u\vert_{X^{s+1}}^2.
$$
Therefore (\ref{estimate5}) follows from the continuous embedding $H^s\subset W^{1,\infty}$
($s>3/2$).
\end{proof}
Similarly, the terms involving $g$ (fourth  line in the r.h.s of (\ref{existenceCH})) are controled using the following lemma:
\begin{lemma}\label{lemma4}
 Under the assumptions of Theorem \ref{th1}, one has that
\begin{eqnarray}
\Big\vert-\mu \big([\Lambda^s,g]\partial_x^2 u-\frac{3}{2}g_x\Lambda^s\partial_x u-g_{xx}\Lambda^s u,\Lambda^s \partial_x u\big)&&\label{estiamte6}\\\quad-\mu\big([\Lambda^s,g_x]\partial_x^2 u,\Lambda^s u\big)\Big\vert
\leq \varepsilon 
 C \vert u\vert_{X^{s+1}}^2.\nonumber
\end{eqnarray}
\end{lemma}
\begin{proof}
Since $\vert u\vert_{H^s}\leq \vert u\vert_{X^{s+1}}$ and $\sqrt{\mu}\vert\partial_x u\vert_{H^s}\leq 
\frac{1}{\sqrt{m}}\vert u\vert_{X^{s+1}}$ , by using the Cauchy-Schwarz inequality  and proceeding as for the proofs of Lemmas \ref{lemma3}, \ref{lemma3'}
one gets directly  (\ref{estiamte6}).
\end{proof}
Finally to control the terms involving $h_i$, (fifth and sixth lines in the r.h.s of (\ref{existenceCH})) let us state the following lemma:
\begin{lemma}\label{lemma5}
 Under the assumptions of  Theorem \ref{th1}, one has that
\begin{eqnarray}
\vert -\varepsilon\mu \big([\Lambda^s,h_1v]\partial_x^2 u-\frac{3}{2}(h_1v)_x\Lambda^s\partial_x u-(h_1v)_{xx}\Lambda^s u,\Lambda^s \partial_x u\big)&&
\label{estiamte7}\\\quad
	-\varepsilon\mu \big([\Lambda^s,(h_1v)_x]\partial_x^2 u,\Lambda^s u\big)\vert
\leq \varepsilon 
 C( \vert v\vert_{X^{s+1}}) \vert u\vert_{X^{s+1}}^2,\nonumber
\end{eqnarray}
\begin{equation}\label{estiamte8}
 \vert-\varepsilon\mu \big(\Lambda^s((h_2v)_x\partial_x u),\Lambda^s \partial_x u\big)\vert\leq \varepsilon 
 C( \vert v\vert_{X^{s+1}}) \vert u\vert_{X^{s+1}}^2.
\end{equation}
\end{lemma}
\begin{proof}
We remark first that 
$$
\vert\mu\big((h_1v)_{xx}\Lambda^s u,\Lambda^s \partial_x u\big)\vert\leq \vert \partial_x(\sqrt{\mu}
\partial_x(h_1v))\vert_{\infty}\vert u\vert_{X^{s+1}}^2
$$
since $s-1>\frac{1}{2}$, so by using the imbedding $H^{s-1}(\R) \subset L^\infty(\R)$ we get 
$$
\vert\partial_x(\sqrt{\mu}
\partial_x(h_1v))\vert_{\infty}\leq C(\vert v\vert_{X^{s+1}}).
$$
 Proceeding now  as for the proof of Lemma  \ref{lemma4}
one gets directly  (\ref{estiamte7}) and (\ref{estiamte8}).
\end{proof}

Gathering the informations provided by the above lemmas, we get
$$
	e^{\varepsilon\lambda t}\partial_t (e^{-\varepsilon\lambda t}\vert u\vert_{X^{s+1}}^2)
	\leq  \varepsilon\big(C(\vert v\vert_{X^{s+1}})-\lambda\big)\vert u\vert_{X^{s+1}}^2+2\varepsilon \vert f\vert_{X^{s+1}}\vert u\vert_{X^{s+1}}.
$$
Taking $\lambda=\lambda_T$ large enough (how large depending on 
$\sup_{t\in [0,\frac{T}{\varepsilon}]}C(\vert v(t)\vert_{X^{s+1}})$
to have the first term of the right hand side negative for all $t\in [0,\frac{T}{\varepsilon}]$, one deduces
$$
	\forall t\in [0,\frac{T}{\varepsilon}],\qquad
	\partial_t (e^{-\varepsilon\lambda_T t}\vert u\vert_{X^{s+1}}^2)\leq 
	2\varepsilon e^{-\varepsilon\lambda_T t}\vert f\vert_{X^{s+1}}\vert u\vert_{X^{s+1}}.
$$
Integrating this differential inequality yields therefore
$$
	\forall t\in [0,\frac{T}{\varepsilon}],\qquad
	\vert u(t)\vert_{X^{s+1}}\leq e^{\varepsilon\lambda_T t}\vert u^0\vert_{X^{s+1}}
	+2\varepsilon \int_0^t e^{\varepsilon\lambda_T (t-t')}
	\vert u(t')\vert_{X^{s+1}}dt'.
$$
Thanks to this energy estimate, one can conclude classically
(see e.g. \cite{AG})
to the existence of 
$$
	T=T(\vert u^0\vert_{X^{s+1}})>0,
$$ 
and of
a unique solution $u\in C([0,\frac{T}{\varepsilon}];X^{s+1}(\R^d))$ to
(\ref{chgeneralsys}) as a limit of the iterative scheme
$$
	u_0=u^0,\quad\mbox{ and }\quad
	\forall n\in\N, \quad
	\left\lbrace\begin{array}{l}
	{\mathcal L}(u^n,\partial)u^{n+1}=0,\\
	u^{n+1}_{\vert_{t=0}}=u^0.
		    \end{array}\right.
$$
Since $u$ solves (\ref{wpchgeneral}), we have ${\mathcal L}(u,\partial)u=0$
and therefore 
$$
	(\Lambda^{s-1}(1-\mu m\partial_x^2)\partial_t u,\Lambda^{s-1}\partial_t u)
	=-(\Lambda^{s-1}{\mathcal M}(u,\partial)u,\Lambda^{s-1}\partial_t u),
$$
with ${\mathcal M}(u,\partial)={\mathcal L}(u,\partial)-(1-\mu m\partial_x^2)\partial_t$.
Proceeding as above, one gets
$$
	\vert \partial_t u\vert_{X^{s}}\leq
	C(\vert u^0\vert_{X^{s+1}}, \vert u\vert_{X^{s+1}}),
$$
and it follows that the family of solution is also bounded in 
$C^1([0,\frac{T}{\varepsilon}];X^s)$.
\end{proof}
\subsection{Explosion condition for the variable bottom CH equation} \label{subsectroisdeux}
As in the case of flat bottoms, it is possible to give some information on the 
 blow-up pattern for the equation 
(\ref{chgbw}) for the free surface.
\begin{proposition}\label{blowCH1}
Let $b \in H^\infty(\R)$, $\zeta_0 \in H^3(\R)$. If the maximal existence time $T_m>0$ of the solution 
of (\ref{chgbw}) with initial profile $\zeta(0,\cdot)=\zeta_0$ is finite,  
the solution $\zeta \in C([0,T_m);H^3(\R)) \cap C^1([0,T_m);H^{2}(\R)$ is such that
\begin{equation}\label{CHb1}
\sup_{t \in [0,T_m),\,x \in \R}\{ |\zeta(t,x)|\} < \infty
\end{equation}
and 
\begin{equation}\label{CHb2}
\sup_{x \in \R}\,\{ \zeta_x(t,x)\} \uparrow \infty\quad as \quad t \uparrow T_m.
\end{equation}
\end{proposition}
\begin{remark}\label{remarkbw'}
It is worth remarking that even though topography effects are introduced in our equation (\ref{chgbw}), 
wave breaking remains of 'surging' type (i.e the slope grows to +$\infty$) as for flat bottoms.
This shows that plunging breakers (i.e the slope decays to -$\infty$) occur for stronger topography variations 
then those considerd in this paper.
\end{remark}
\begin{proof}
By using the Theorem \ref{th1} given $\zeta_0 \in H^3(\R)$, the maximal 
existence time of the solution $\zeta(t)$ to (\ref{chgbw}) with initial data $\zeta(0)=\zeta_0$ 
is finite if and only if $\vert \zeta(t)\vert_{H^3(\R)}$ blows-up in finite time. To complete the proof it is enough to show 
that 

(i) the solution $\zeta(t)$ given by Theorem \ref{th1} remains uniformly bounded as long as it is defined;

\noindent
and 

(ii) if we can find some $M=M(\zeta_0)>0$ with  
\begin{equation}\label{CHb3}
\zeta_x(t,x) \le M,\qquad x \in \R,
\end{equation}
as long as the solution is defined, then $\vert \zeta(t)\vert_{H^3(\R)}$ remains bounded on $[0,T_m)$ .\\ 
Remarking that
$$
\int_\R [(c\zeta_x+\frac{1}{2}c_x\zeta)\zeta]\,dx=-\int_\R [(c\zeta)_x\zeta-
\frac{1}{2}c_x\zeta^2]\,dx=-\int_\R [(c\zeta_x+\frac{1}{2}c_x\zeta)\zeta]\,dx,
$$
we can deduce
\begin{equation}\label{id0}
 \int_\R [(c\zeta_x+\frac{1}{2}c_x\zeta)\zeta]\,dx=0.
\end{equation}
 We have also that 
\begin{equation}\label{id10}
 \int_\R[(\zeta^i\zeta_{x})\zeta]\,dx=0, \qquad \forall i\in \N^*.
\end{equation}
 Item (i) follows at once from (\ref{id0}), (\ref{id10}) and the imbedding $H^1(\R) \subset L^\infty(\R)$  since multiplying (\ref{chgbw}) 
by $\zeta$ and integrating on $\R$ yields
\begin{equation}\label{CHb4}
\partial_t\Big(\int_\R [\zeta^2+\frac{1}{12}\,\mu\int_\R \zeta_x^2]\,dx\Big)=0.
\end{equation}

To prove item (ii), as in \cite{AD:08} we  multiply the equation (\ref{chgbw}) by $\zeta_{xxxx}$ 
and after performing several integrations by parts we obtains the following identity:
\begin{eqnarray}\label{CHb5}
&&\partial_t\Big( \int_\R [\zeta_{xx}^2+\frac{1}{12}\,\mu\int_\R \zeta_{xxx}^2]\,dx\Big) =15\,\varepsilon\int_{\R}\zeta\zeta_{xx}\zeta_{xxx}\,d-\frac{15}{4}
\,\varepsilon^2 \int_{\R} \zeta^2\zeta_{xx}\zeta_{xxx}\,dx \\
&&\qquad
 +\frac{9}{16}\varepsilon^3 \int_\R \zeta_x^5\,dx
+\frac{15}{8}\,\varepsilon^3\int_\R\zeta^3\zeta_{xx}\zeta_{xxx}\,dx+\frac{7}{4}\,\mu\varepsilon\int_\R \zeta_x\zeta_{xxx}^2\,dx+I+J.\nonumber
\end{eqnarray}
Where, 
$$
I=\int_\R c\zeta_x\zeta_{xxxx}\,dx \qquad
J=\frac{1}{2}\int_\R c_x\zeta\zeta_{xxxx}\,dx.
$$
One can use the Caushy-schwarz inequality to get
\begin{eqnarray}\label{I}
 &&I=-\int_\R \zeta_x(c\zeta_x)_{xxx}=-\frac{1}{2}\int_\R c_{xxx}\zeta_x^2-\frac{3}{2}\int_\R
c_{xx}\zeta_{xx}\zeta_x-\frac{3}{2}\int_\R c_x\zeta_{xxx}\zeta_x\\\
&&\quad\leq   M_1\big(\frac{1}{2}\vert \zeta_x\vert_2 \vert\zeta_{x}\vert_2 +\frac{3}{2}\vert \zeta_{xx} \vert_2 \vert\zeta_{x} \vert_2+\frac{3}{2}\vert \zeta_{xxx} \vert_2 \vert\zeta_{x} \vert_2\big), \nonumber
\end{eqnarray}
and
\begin{eqnarray}\label{J}
 &&J=-\frac{1}{2}\int_\R c_{xx}\zeta\zeta_{xxx}-\frac{1}{2}\int_\R c_x\zeta_{x}\zeta_{xxx}\\
&&\quad\leq M_1\big(\frac{1}{2}\vert \zeta\vert_2 \vert\zeta_{xxx}\vert_2 + \frac{1}{2} \vert \zeta_{x} \vert_2 \vert\zeta_{xxx} \vert_2\big) ,\nonumber
\end{eqnarray}
for some $M_1=M_1(\vert c\vert_{W^{3,\infty}})>0.$
If (\ref{CHb3}) holds, let in accordance with (\ref{CHb4})  $M_0>0$ be such that 
\begin{equation}\label{zeta}
|\zeta(t,x)| \le M_0,\qquad x \in \R,
\end{equation}
for as long as the solution exists. Using the Cauchy-Schwarz inequality 
as well as the fact that $\mu \le 1$, we infer from (\ref{CHb4}), (\ref{CHb5}), (\ref{I}) , (\ref{J}) and (\ref{zeta}) that there exists $C(M_0,M_1,M,\varepsilon,\mu)$ such that 
$$\partial_t E(t)\le  C(M_0,M_1,M,\varepsilon,\mu)\, E(t),$$
where
$$E(t)=\int_\R [\zeta^2+ \frac{1}{12}\,\mu\zeta_x^2+ \zeta_{xx}^2 +\frac{1}{12}\,\mu
 \zeta_{xxx}^2]\,dx.$$
(Note that we do not give any details for the components of  (\ref{CHb5}) other than $I$ and $J$ because these components do not involve any topography term and can therefore be handled exactly as in \cite{AD:08}).
An application of Gronwall's inequality enables us to conclude.
\end{proof}
Our next aim is to show as in the case of flat bottoms that there are solutions to (\ref{chgbw}) that blow-up in finite time as 
surging breakers, that is, following the pattern given in Proposition \ref{blowCH1}. 
\begin{proposition}\label{blowCH2}
Let $b \in H^\infty(\R)$. If the initial wave profile $\zeta_0 \in H^3(\R)$ satisfies
\begin{eqnarray*}
\Big|\sup_{x \in \R} \,\{\zeta_0(x)\}\Big|^2 &\ge& \frac{28}{3}\,C_0\,\mu^{-3/4}
+\frac{1}{2}\,\varepsilon\,C_0^{3/2}\,\mu^{-3/4}+\frac{1}{4}\,\varepsilon^2\,C_0^2\,\mu^{-3/4}\\
&&\qquad +\frac{7}{3}\,C_0\,\mu^{-1/2}+\frac{8}{3}\,C_0^{1/2}\,C_1\,\mu^{-3/4}\,\varepsilon^{-1}
+\frac{4}{3}\,C_0^{1/2}\,C_1\,\mu^{-3/4}\,\varepsilon^{-1},
\end{eqnarray*}
where
$$C_0=\int_\R [\zeta_0^2+\frac{1}{12}\mu(\zeta_0')^2]\,dx>0,\quad C_1=\vert c \vert_{W^{1,\infty}}>0$$
then wave breaking occurs for the solution of (\ref{chgbw}) in finite time $T=O(\frac{1}{\varepsilon})$.
\end{proposition}
\begin{proof}
One can adapt the proof of this Proposition in the same way of the proof  of the Proposition 6 in \cite{AD:08}, and we omit the proof here.
\end{proof}
\subsection{Well-posedness for the variable bottom KDV equation} \label{subsectroistrois}
We prove now the well-posedness of the equations (\ref{kdv-top}), (\ref{kdv-topV}).
We consider the following  general class of equations
\begin{equation}\label{kdv-topwp}
	\left\vert
	\begin{array}{l}
       u_t + cu_x +kc_xu+ \varepsilon guu_x + 
        \frac{1}{6}\mu c^5 u_{xxx} =0
	\\
	u_{\vert_{t=0}}=u^0,
	\end{array}\right.
\end{equation}
where $k\in \R$, $g=\frac{3}{2}$ for (\ref{kdv-topV}) and $g=\frac{3}{2c}$ for (\ref{kdv-top}).
This class of equations is not included in the family of equations stated in \S 3.1 because $m=0$ (here there is not $\partial_x^2\partial_tu$ term).
\begin{theorem}\label{th2}
        Let $s>\frac{3}{2}$ and $b\in H^{\infty}(\R)$. Let also $\wp'$  be a  family of parameters   $\theta=(\varepsilon,\beta,\alpha,\mu)$
	 satisfying (\ref{cond2}). Assume morever that
$$
	\exists c_0>0,\quad \forall\;\theta\; \in\;\wp', \qquad
	c(x)=\sqrt{1-\beta b^{(\alpha )}(x)}\ge c_0.
$$
        Then for all $u^0 \in H^{s+1}(\R)$, there exists $T>0$ and a unique family of solutions 
	$(u^{\theta})_{\theta\in{\wp}}$ to (\ref{kdv-topwp})
	bounded in $C([0,\frac{T}{\varepsilon}];H^{s+1}(\R))\cap
	C^1([0,\frac{T}{\varepsilon}];H^{s}(\R))$.
\end{theorem}
\begin{proof}
As in the proof of Theorem \ref{th1}, for all $v$ smooth enough, let us define the ``linearized'' operator
${\mathcal L}(v,\partial)$ as:
\begin{eqnarray*}
	{\mathcal L}(v,\partial)&=&\partial_t
	+c\partial_x+kc_x+\varepsilon gv\partial_x+\frac{1}{6}\mu c^5 \partial_x^3.
\end{eqnarray*}
We define now  the initial value problem as:
\begin{equation}\label{linsyskdv}
	\left\lbrace
	\begin{array}{l}
	{\mathcal L}(v,\partial)u=\varepsilon f,\\
	u_{\vert_{t=0}}=u^0.
	\end{array}\right.
\end{equation}
Equation (\ref{linsyskdv}) is a linear equation which can be solved in any interval of time in which the coefficients are defined.
We establish some precise energy estimates on the solution. First remark that when $m=0$, the energy norm $\vert\cdot\vert_{X^s}$ defined in the proof
of Theorem \ref{th1} does not allow to control the term $gu_{xxx}$ (for instance). Indeed, Lemma \ref{lemma4} requires $m \not= 0$ to be true. We show that 
a control of this term is however possible if we use an adequate weight function to defined the energy
and use the dispersive properties of the equation. More
precisely, inspired by \cite{craigkdv} let us define the ``energy'' norm for all  $s\ge 0$ as:
$$
	E^s(u)^2=\vert w\Lambda^s u\vert_{2}^2
$$
where the weight function $w$ will be determined later. For the moment, we just require that 
 there exists  two positive numbers $w_1,w_2$
such that $\forall\;x\;\in\;\R$
$$
w_1\leq w(x)\leq w_2,
$$
so that $E^s(u)$ is uniformly equivalent to the standard $H^s$-norm.
Differentiating $\frac{1}{2}e^{-\varepsilon\lambda t}E^s(u)^2$
with respect to time, one gets using  (\ref{linsyskdv})
\begin{eqnarray*}
	\lefteqn{\frac{1}{2}e^{\varepsilon\lambda t}\partial_t(e^{-\varepsilon\lambda t} E^s(u)^2)
	=-\frac{\varepsilon\lambda}{2} E^s(u)^2 -\big([\Lambda^s,c]\partial_xu,w\Lambda^s u\big)-
         (c\partial\Lambda^su ,w\Lambda^s u\big)}\\
	&&-k(\Lambda^s (u\partial_xc ),w\Lambda^su)-\varepsilon \big([\Lambda^s,gv]\partial_x u,w\Lambda^s u)
	-\varepsilon\big(gv\partial\Lambda^s u,w\Lambda^s u\big)\\
	& &-\frac{1}{6}\mu \big([\Lambda^s,c^5]\partial_x^3 u,w\Lambda^s  u\big)-\frac{1}{6}\mu\big(c^5\partial^3\Lambda^s u,w\Lambda^s u\big)
        +\varepsilon\big(\Lambda^sf ,w\Lambda^su\big).\\
\end{eqnarray*}
It is clear, by a simple integration by parts that
\begin{equation}\label{control1}
\vert \varepsilon\big(gv\partial\Lambda^s u,w\Lambda^s u\big)\vert\leq \varepsilon C(\vert v \vert_{W^{1,\infty}},\vert w\vert_{W^{1,\infty}}) E^s(u)^2.
\end{equation}
Let us now focus on the seventh and eighth terms of the right hand side of the previous identity. 
In order to get an adequate control of the seventh term, we have to write explicitly the commutator $[\Lambda^s,c^5]$:
$$
[\Lambda^s,c^5]\partial_x^3u=\{\Lambda^s,c^5\}_2\partial_x^3u+Q_1\partial_x^3u,
$$
where $\{\cdot,\cdot\}_2$ stands for the second order Poisson brackets, 
$$
\{\Lambda^s,c^5\}_2=-s\partial_x(c^5)\Lambda^{s-2}\partial_x+\frac{1}{2}[s\partial_x^2(c^5)\Lambda^{s-2}-
s(s-2)\partial_x^2(c^5)\Lambda^{s-4}\partial_x^2]$$
and $Q_1$  is an operator of order $s-3$ that can be controled by the general commutator estimates of  (see \cite{lannes'}). We thus get
$$
\vert \big(Q_1\partial_x^3u,w\Lambda^s u\big)\vert\leq C \varepsilon E^s(u)^2.
$$
We now use the identity $\Lambda^2=1-\partial_x^2$ and the fact that 
$\alpha\beta=O(\varepsilon)$, to get, 
as in (\ref{control1}),
$$
\vert \big([s\partial_x^2(c^5)\Lambda^{s-2}-
s(s-2)\partial_x^2(c^5)\Lambda^{s-4}\partial_x^2]\partial_x^3u,w\Lambda^s u\big)\vert\leq \varepsilon C(\vert w\vert_{W^{1,\infty}})E^s(u)^2.
$$
This leads to the expression
$$
\frac{1}{6}\mu \big([\Lambda^s,c^5]\partial_x^3 u,w\Lambda^su\big)=
\frac{s}{6}\mu \big(\partial c^5 \Lambda^{s}\partial^2 u,w\Lambda^s u\big)+Q_2,
$$
where $\vert Q_2 \vert \leq \varepsilon C(\vert w\vert_{W^{1,\infty}} )E^s(u)^2$. 
Remarking now that 
\begin{eqnarray*}
&&\frac{s}{6}\mu \big(\partial c^5\Lambda^{s}\partial^2 u ,w\Lambda^s u\big)
=\\&&\qquad-\frac{s}{6}\mu \big(\partial (\partial(c^5)w) \Lambda^{s} \partial u,\Lambda^s u\big)-
\frac{s}{6}\mu \big(\partial (c^5) w,(\Lambda^s \partial u)^2\big).
\end{eqnarray*}
The control of the eighth term comes in the same way:
\begin{eqnarray*}
&&\frac{1}{6}\mu\big(c^5\partial^3\Lambda^s u,w\Lambda^s u\big)=-
\frac{1}{12}\mu\big(\partial^3(c^5w)\Lambda^s u,\Lambda^s u\big)\\&&\qquad-
\frac{1}{4}\mu\big(\partial^2(wc^5)\Lambda^{s}\partial u,\Lambda^s u\big)-
\frac{1}{4}\mu\big(\partial(wc^5)\Lambda^s u,\Lambda^{s} \partial^2 u\big)
\end{eqnarray*}
similarly:
\begin{eqnarray*}
&&-\frac{1}{4}\mu\big(\partial(wc^5)\Lambda^s u,\Lambda^{s}\partial^2 u\big)=\\&&\qquad
\frac{1}{4}\mu \big(\partial^2(c^5w) \Lambda^{s}\partial u,\Lambda^s u\big)+
\frac{1}{4}\mu \big(\partial(c^5w),(\Lambda^{s}\partial u)^2\big).
\end{eqnarray*}
We choice here $w$ so that
\begin{equation}\label{eqw}
-\frac{s}{6}\mu \big(\partial (c^5) w,(\Lambda^s \partial u)^2\big)+\frac{1}{4}\mu \big(\partial(c^5w),(\Lambda^{s}\partial u)^2\big)=0
\end{equation}
therefore if one take $w= c^{5({\frac{2s}{3}-1)}}$ we get eaisly (\ref{eqw}). Finally, one has
\begin{eqnarray*}
 &&\frac{1}{6}\mu \big([\Lambda^s,c^5]\partial_x^3 u,w\Lambda^s  u\big)+\frac{1}{6}\mu\big(c^5\partial^3\Lambda^s u,w\Lambda^s u\big)\\
&&=Q_2-\frac{s}{6}\mu \big(\partial (\partial(c^5)w) \Lambda^s\partial u,\Lambda^s u\big)
-\frac{1}{12}\mu\big(\partial^3(c^5w)\Lambda^s u,\Lambda^s u\big)\\&&\qquad-
\frac{1}{4}\mu\big(\partial^2(c^{5}w)\Lambda^{s}\partial u,\Lambda^s u\big)+
\frac{1}{4}\mu \big(\partial^2(c^{5}w)\Lambda^{s}\partial u,\Lambda^su\big),
\end{eqnarray*}
using again the fact that $\alpha\beta=O(\varepsilon)$ one can deduce
$$
e^{\varepsilon\lambda t}\partial_t (e^{-\varepsilon\lambda t}E^s(u)^2) 
	\leq  \varepsilon\big(C(E^s(v))-\lambda\big)E^s(u)^2+4\varepsilon E^s(f)E^s(u).
$$
This inequality, together with end of the proof of Theorem \ref{th1}, easily yields the result.
\end{proof}
\section{Rigorous justification of the variable bottom equations}\label{sectquatre}
\subsection{Rigorous justification of the variable bottom CH equation} \label{subsectquatreun}
We restrict here our attention to parameters 
 $\varepsilon$, $\beta$, $\alpha$ and $\mu$ linked by the relations  
\begin{eqnarray}
\label{cond3}
&\varepsilon=O(\sqrt{\mu}),\;\beta\alpha=O(\varepsilon),\;
 \beta\alpha=O(\mu^2).
\end{eqnarray}
These conditions are stronger than (\ref{cond1}), and this allows us to control the secular effects that prevented us from proving an $H^s$-consistency (and \emph{a fortiori} a full justification) for the variable bottom
equations of \S \ref{sectdeux}. 
The notion of  $H^s$-consistency is defined below:
\begin{definition}\label{defiH^s}
 Let $\wp_1$ be a family of parameters $\theta=(\varepsilon,\beta,\alpha,\mu)$ satisfying (\ref{cond3}). A family $(\zeta^{\theta},u^{\theta})_{\theta\in \wp_1}$ is 
$H^s$-consistent of order $s\ge0$ and on $[0,\frac{T}{\varepsilon}]$ with the GN equations (\ref{GN3}), if for all $\theta \in \wp_1$, (and denoting $h^\theta=1+\varepsilon\zeta^\theta-\beta b^{(\alpha)}$),\\
$$\left\lbrace
    \begin{array}{l}
\zeta^{\theta}_t +[h^\theta u]_x=\mu^{2}r_{1}^{\theta}\\
u^{\theta}_t+\zeta^{\theta}_x +\varepsilon u^{\theta}u^{\theta}_x=\displaystyle\frac{\mu}{3}\frac{1}{h^\theta}
[(h^\theta)^3(u^{\theta}_{xt}+\varepsilon u^{\theta}u^{\theta}_{xx}-\varepsilon (u^{\theta}_x)^2)]_x+\mu^{2}r_{2}^{\theta}
 \end{array}
\right.$$\\
with $(r_{1}^{\theta},r_{2}^{\theta})_{\theta\in \wp_1}$ bounded in 
$L^{\infty}([0,\frac{T}{\varepsilon}],H^s(\R)^2).$
\end{definition}
\begin{remark}
Since $H^s(\R)$ is continously embedded in $L^{\infty}(\R)$ for $s>1/2$, the $H^s$-consistency implies the $L^{\infty}$-consistency when $s>1/2$.
\end{remark}
\begin{proposition}\label{propH^sv}
 Let $b\in H^{\infty}(\R)$ and $p\in \R$. Assume that 
$$A=p, \quad B=p-\frac{1}{6} \quad E=-\frac{3}{2}p-\frac{1}{6},\quad F=-\frac{9}{2}p-\frac{23}{24}.$$
Then:
\begin{itemize}
\item For all family $\wp_1$ of parameters satisfying (\ref{cond3}),
 \item For all $s\ge0$ large enough  and $T>0$, 
\item For all bounded family $(u^{\theta})_{\theta\in \wp_1} \in C([0,\frac{T}{\varepsilon}],H^{s}(\R))$ solving (\ref{chgeneral}),
\end{itemize}
the familly $(\zeta^{\theta},u^{\theta})_{\theta\in \wp_1}$ with (omitting the index $\theta$)\\
\begin{equation}\label{zetaexpH^s}
 \zeta := cu+\frac{\varepsilon}{4}u^{2}+\frac{\mu}{6}c^4u_{xt}-
\varepsilon\mu c^4[\frac{1}{6}uu_{xx}+\frac{5}{48}u_x^{2}],
\end{equation}
is $H^s$-consistent on $[0,\frac{T}{\varepsilon}]$ with the GN equations (\ref{GN3}).
\end{proposition}
\begin{proof}
This is clear by using the same arguments of the proof of Proposition \ref{propvelocity} if we remark that now the term $c_xu$ (responsible of the secular growth effects) is of order $O(\mu^2)$ in $L^{\infty}([0,\frac{T}{\varepsilon}],H^s(\R))$. 
Therefore  this quantity is not required in the identity (\ref{eqon v}) that defines $v$. In the proof of Proposition \ref{propvelocity}  there is therefore no  $c_xu$ term  and 
 the $O_{L^{\infty}}(\mu^2)$ terms in Proposition \ref{propvelocity} are now of order $O(\mu^2)$.
\end{proof}
In Proposition \ref{propH^sv} , we constructed a family 
$(u^{\theta},\zeta^{\theta})$ \emph{consistent} with the
Green-Naghdi equations in the sense of Definition \ref{defiH^s}.
A consequence of the following theorem is a stronger result: this family
provides a good approximation of the exact solutions 
$(\underline{u}^{\theta},\underline{\zeta}^{\theta})$ of the
Green-Naghdi equations with same initial data in the sense that
$(\underline{u}^{\theta},\underline{\zeta}^{\theta})=
(u^{\theta},\zeta^{\theta})+O(\mu^2 t)$ for times $O(1/\varepsilon)$.
\begin{theorem}\label{th3}
	Let $b\in H^{\infty}(\R)$, $s\ge 0$ and $\wp_1$ 
	be a family of parameters satisfying (\ref{cond3}) with $\beta=O(\varepsilon)$.
	Let also 
	$A$, $B$, $E$ and $F$ be as in Proposition \ref{propH^sv}. 
	If $B<0$ then
	there exists $D>0$, $P>D$ and $T>0$ such that for all 
	$u^{\theta}_0\in H^{s+P}(\R)$:
	\begin{itemize}
	\item There is a unique family 
	$(u^{\theta},\zeta^{\theta})_{\theta\in \wp_1}\in 
	C([0,\frac{T}{\varepsilon}];H^{s+P}(\R)\times H^{s+P-2})$ given by 
	the resolution of (\ref{chgeneral}) 
	with initial condition $u^{\theta}_{0}$ and formula (\ref{zetaexpH^s});
	\item  There is a unique family
	$(\underline{u}^{\theta},\underline{\zeta}^{\theta})_{\theta
        \in\wp_1}\in 
	C([0,\frac{T}{\varepsilon}];H^{s+D}(\R)^2)$ solving the Green-Naghdi equations
	(\ref{GN3}) with initial condition 
	$(u^{\theta}_0,\zeta^{\theta}_{\vert_{t=0}})$.
	\end{itemize}
	Moreover, one has for all $\theta\in \wp_1$,
	$$
	\forall t\in [0,\frac{T}{\varepsilon}],\qquad
	\vert \underline{u}^{\theta}-u^{\theta}\vert_{L^\infty([0,t]\times \R)}+\vert \underline{\zeta}^{\theta}-\zeta^{\theta}\vert_{L^\infty([0,t]\times \R)}\leq \cst\; \mu^2 t.
	$$
\end{theorem}
\begin{remark}
        It is known (see \cite{AL}) that the Green-Naghdi equations give,
	under the scaling (\ref{scalingch}) with $\beta=O(\varepsilon)$ 
        a correct approximation of 
	the exact solutions of the full water-waves equations (with a precision
	$O(\mu^2 t)$ and over a time scale $O(1/\varepsilon)$). It follows 
	that the unidirectional approximation discussed above
	approximates the solution of the water-waves equations
	with the same accuracy.
\end{remark}
\begin{remark}
        We used the unidirectional equations derived on the velocity
	as the basis for the approximation justified in the Theorem \ref{th3}.
	One could of course use instead the unidirectional approximation
	(\ref{chgeneralzeta}) derived on the surface elevation.
\end{remark}
\begin{proof}
        The first point of the theorem is a direct consequence of Theorem \ref{th1}.
        Thanks to Proposition \ref{propH^sv}, we know that 
        $(u^{\theta},\zeta^{\theta})_{\theta\in\wp_1}$ is $H^s$-consistent with the Green-Naghdi 
        equations (\ref{GN3}), so that the second point of the theorem and the error estimate follow
        at once from the well-posedness and stability of the Green-Naghdi
        equations under the present scaling (see Theorem 4.10 of \cite{AL2}).
\end{proof}
\subsection{Rigorous justification of the variable bottom KDV-top equation} \label{subsectquatredeux}
In this subsection the parameters 
 $\varepsilon$, $\beta$, $\alpha$ and $\mu$ are assumed to satisfy
\begin{eqnarray}
\label{cond4}
&\varepsilon=O(\mu),\;\beta\alpha=O(\varepsilon^2).
\end{eqnarray}
We give first a Proposition regarding the $H^s$-consistency result for the KDV-top equation.
\begin{equation}
 \zeta_t + c\zeta_x +\frac{3}{2c} \varepsilon \zeta\zeta_x + 
\frac{1}{6}\mu c^5 \zeta_{xxx} +\frac{1}{2}c_x\zeta=0,
\label{kdv-top'}
\end{equation}
\begin{proposition}\label{propboussH^s}
 Let $b\in H^{\infty}(\R)$. Then:
\begin{itemize}
\item For all family $\wp_1'$ of parameters satisfying (\ref{cond4}),
 \item For all $s\ge 0$ large enough and $T>0$, 
\item For all bounded family $(\zeta^{\theta})_{\theta\in \wp_1'} \in C([0,\frac{T}{\varepsilon}],H^{s}(\R))$ solving (\ref{kdv-top'})
\end{itemize}
the familly $(\zeta^{\theta},u^{\theta})_{\theta\in \wp'}$ with (omitting the index $\theta$)\\
\begin{equation}\label{velocityexpkdvH^s}
u:= \frac{1}{c}\Big(\zeta-\frac{\varepsilon}{4c^2}\zeta^2
	+\mu\frac{1}{6}c^4\zeta_{xx}\Big)
\end{equation}
is $H^s$-consistent on $[0,\frac{T}{\varepsilon}]$ with the equations (\ref{GN3}).
\end{proposition}
\begin{remark}\label{remarkVH^s}
Similarly, one can prove that the solution of the  following KDV-top equations
\begin{equation}\label{kdv-topVH^s}
 u_t + cu_x +\frac{3}{2}\varepsilon uu_x + 
\frac{1}{6}\mu c^5 u_{xxx} +\frac{3}{2}c_xu=0,
\end{equation}
on the velocity, with 
\begin{equation}\label{zetaexpkdvH^s}
 \zeta := cu+\frac{\varepsilon}{4}u^{2}-\frac{\mu}{6}c^5u_{xx}
\end{equation}
 is $H^s$-consistent with the equations Green-Naghdi equations (\ref{GN3}).
\end{remark}
\begin{proof}
One can adapt the proof of Proposition \ref{propbouss} in the same way as we adapted the proof of Proposition \ref{propvelocity} to establish Proposition
\ref{propH^sv} (we also use the fact that if a family is consistent with the Boussinesq equations (\ref{BOUSS}), it is also consistent with the Green-Naghdi equations (\ref{GN3}) under the present scaling).
\end{proof}
The following Theorem deals the rigorous justification of the KDV variable bottom equation:
\begin{theorem}\label{th4}
	Let $b\in H^{\infty}(\R)$, $s\ge 0$ and $\wp'_1$ 
	be a family of parameters satisfying (\ref{cond4}) with $\beta=O(\varepsilon)$. 
	If $B<0$ then
	there exists $D>0$, $P>D$ and $T>0$ such that for all 
	$\zeta^{\theta}_0\in H^{s+P}(\R)$:
	\begin{itemize}
	\item There is a unique family 
	$(\zeta^{\theta},u^{\theta})_{\theta\in \wp'_1}\in 
	C([0,\frac{T}{\varepsilon}];H^{s+P-2}(\R)\times H^{s+P})$ given by 
	the resolution of (\ref{kdv-top'}) 
	with initial condition $\zeta^{\theta}_{0}$ and formula (\ref{velocityexpkdvH^s});
	\item  There is a unique family
	$(\underline{\zeta}^{\theta},\underline{u}^{\theta})_{\theta
        \in\wp'_1}\in 
	C([0,\frac{T}{\varepsilon}];H^{s+D}(\R)^2)$ solving the  Green-Naghdi equations
	(\ref{GN3})  with initial condition 
	$(\zeta^{\theta}_0,u^{\theta}_{\vert_{t=0}})$.
	\end{itemize}
	Moreover, one has for all $\theta\in \wp'_1$,
	$$
	\forall t\in [0,\frac{T}{\varepsilon}],\qquad
	\vert \underline{u}^{\theta}-u^{\theta}\vert_{L^\infty([0,t]\times \R)}+\vert \underline{\zeta}^{\theta}-\zeta^{\theta}\vert_{L^\infty([0,t]\times \R)}\leq \cst\; \varepsilon^2 t.
	$$
\end{theorem}
\begin{remark}
        We used the unidirectional equation derived on the free surface elevation
	as the basis for the approximation justified in the Theorem \ref{th4}.
	One could of course use instead the unidirectional approximation
	(\ref{kdv-topVH^s}) derived on velocity.
\end{remark}
\begin{proof}
        The first point of the theorem is a direct consequence of Theorem \ref{th2}.
        Thanks to Proposition \ref{propboussH^s}, we know that 
        $(u^{\theta},\zeta^{\theta})_{\theta\in\wp'_1}$ is $H^s$-consistent with the Green-Naghdi
        equations (\ref{GN3}), so that the second point follows as in Theorem \ref{th3}.
\end{proof}
\subsection*{Acknowledgments}\ The author is grateful to David Lannes for encouragement and many helpful discussions.

\providecommand{\href}[2]{#2}
\end{document}